\documentclass{amsart}
\usepackage{amsmath,amssymb,amsthm,mathrsfs}
\usepackage{tikz}
\usepackage[colorlinks=true, linkcolor=red, citecolor=green]{hyperref}

\newtheorem{theorem}{Theorem}[section]
\newtheorem{proposition}[theorem]{Proposition}
\newtheorem{lemma}[theorem]{Lemma}
\newtheorem{corollary}[theorem]{Corollary}
\theoremstyle{definition}
\newtheorem{definition}[theorem]{Definition}
\newtheorem{convention}[theorem]{Convention}
\theoremstyle{remark}
\newtheorem{remark}[theorem]{Remark}

\newcommand{\R}{\mathbb{R}} 
\newcommand{\T}{\mathbb{T}}
\newcommand{\C}{\mathbb{C}} 
\newcommand{\N}{\mathbb{N}} 
\newcommand{\Z}{\mathbb{Z}}
\newcommand{\D}{\mathbb{D}}

\newcommand{\K}{\mathcal{K}}

\newcommand{\Ks}{\mathcal{K}^{*}}
\newcommand{\U}{\mathcal{U}}
\newcommand{\Lp}{L^2_+}
\DeclareMathOperator{\spr}{spr}

\newcommand{\bk}{\backslash}
\newcommand{\p}{\partial}
\newcommand{\eps}{\varepsilon}
\newcommand{\la}{\lambda}

\DeclareMathOperator{\Id}{Id}
\DeclareMathOperator{\Range}{Range}

\newcommand{\qtq}[1]{\quad\text{#1}\quad}

\numberwithin{equation}{section}

\begin{document}

\title[Complete integrability of (CS) on $L^2_+(\mathbb T)$]
      {Complete integrability of the Calogero--Sutherland DNLS equation on $L^2_+(\mathbb T)$}

\date{\today}
\author{Rana Badreddine}
\address{School of Mathematics and Maxwell Institute, University of Edinburgh, Edinburgh EH9 3FD, Scotland, United Kingdom}
\email{rbadredd@ed.ac.uk}

\begin{abstract}
We establish the complete integrability of the defocusing Calogero--Sutherland derivative nonlinear Schr\"odinger equation 
\begin{equation}\label{CS-abstract}\tag{CS}
    i\partial_t u +\partial_x^2 u + 2iu \partial_x\,\Pi\!\left(|u|^2\right)=0,\qquad x\in\mathbb T,
\end{equation}
by constructing a nonlinear Fourier transform
\[
   \Phi:u\in L^2_+(\mathbb T)\mapsto\ (\zeta_n(u))_{n\ge0}\in \ell^2(\mathbb N_0),\qquad
   \zeta_n(u)=\sqrt{\gamma_n(u)}\;e^{\,i\arg\langle u\mid f_n\rangle},
\]
built from the eigenvalues $(\lambda_n)$, the spectral gaps $\gamma_n:=\la_n-\la_{n-1}-1\geq 0$ and the eigenfunctions $(f_n)$ of the Lax operator $L_u$. We prove that $\Phi$ is a norm-preserving homeomorphism. In these coordinates, the flow becomes the explicit rotation $\zeta_n(t)=e^{-i\omega_n t}\zeta_n(0)$ for all $n$, whose frequencies $\omega_n$ depend only on the conserved spectrum $(\lambda_n)$. As consequences, every orbit of \eqref{CS-abstract} is precompact in $L^2_+(\mathbb T)$, and $t\mapsto u(t)$ is Bohr almost periodic.
For finite-gap data, we obtain quasi-periodicity in time 
and a sharp criterion for genuine periodicity.
\end{abstract}



\maketitle

\tableofcontents

\section{Introduction}\label{S:1}

This paper investigates the defocusing Calogero--Sutherland derivative nonlinear Schr\"odinger equation on the torus $\T=\R/2\pi\Z$,
\begin{equation*}\label{CS}\tag{CS}
   i\p_t u +\p_x^2 u + 2iu \p_x\,\Pi\!\left(|u|^2\right)=0,
\end{equation*}
for a complex-valued unknown $u(t,\cdot)$ in the Hardy space
\[
   L^2_+(\T)=\bigl\{f\in L^2(\T):\ \widehat f(k)=0\ \text{for }k<0\bigr\},
\]
where $\Pi$ denotes the Riesz--Cauchy--Szeg\H{o} projection of $L^2(\T)$ onto $L^2_+(\T)$
\[
\Pi : L^2(\mathbb{T}) \to L^2_+(\mathbb{T}), \qquad \Pi f(x) = \sum_{k \ge 0} \widehat{f}(k) e^{ikx},
\]
The precise conventions are collected in \S\ref{S:2}. Equation \eqref{CS} is also referred to as the Calogero--Moser derivative NLS, or as a continuum Calogero--Moser model.

From the point of view of scaling, \eqref{CS} is mass-critical: on the line the transformation
\[
   u(t,x)\ \longmapsto\ u_\lambda(t,x):=\sqrt{\lambda}\,u(\lambda^2 t,\lambda x),
   \qquad \lambda>0,
\]
maps solutions to solutions and preserves the mass $\int|u|^2\,dx$, so that $L^2$ is the scaling-critical regularity.  In this critical space, global well-posedness of the defocusing equation on $\T$ was established by the author in \cite{badreddine2024global}, with the corresponding statement on the line $(x\in\R)$ due to Killip--Laurens--Visan \cite{killip2025scaling}. These results were subsequently revisited on both geometries by Killip--Marsden--Visan \cite{KMV} through the Hamiltonian formulation and the method of commuting flows.  

\medskip

Equation \eqref{CS} is completely integrable. On the torus, the Lax-pair formalism underlying our analysis was developed in \cite{badreddine2024global,badreddine2024traveling}. There, \eqref{CS} is realized through a self-adjoint Lax operator $L_u=-i\p_x+u\Pi(\bar u\,\cdot)$ acting on $H^1_+(\T):=H^1\cap L^2_+$, together with a skew-adjoint operator $P_u$, so that \eqref{CS} is equivalent to the Lax equation $\p_t L_{u(t)}=[P_{u(t)},L_{u(t)}]$ (see \eqref{Lax op}--\eqref{Lax eq}). In particular the spectrum of $L_u$ is conserved by the flow, and it consists (Theorem~\ref{thm:Lax}) of a simple, discrete sequence of eigenvalues 
\[
   0\le\lambda_0(u)<\lambda_1(u)<\lambda_2(u)<\cdots\to+\infty,
\]
whose consecutive spacings define the nonnegative \emph{gaps}
\[
\gamma_n(u):=\lambda_n(u)-\lambda_{n-1}(u)-1\ge0, \qquad n\geq 1, \qtq{with } \gamma_0(u):=\lambda_0(u).
\]
  A gap is \emph{open} when $\gamma_n>0$ and closed when $\gamma_n=0$.

For a completely integrable Hamiltonian equation, the most complete structural description one can hope for is a global system of \emph{Birkhoff coordinates}, i.e.\ global action--angle variables, in which the actions are conserved and the flow reduces to a linear rotation. For the Benjamin--Ono equation on the torus such coordinates were constructed by G\'erard--Kappeler \cite{gerardkappeler} and then by G\'erard--Kappeler--Topalov \cite{gerard2019flow}, who showed that the Birkhoff map is a  diffeomorphism onto a weighted sequence space linearizing the BO flow. A comparable picture was obtained for the cubic Szeg\H{o} equation by G\'erard--Grellier \cite{gerardgrellier}. The purpose of the present paper is to construct such global Birkhoff coordinates for the defocusing equation \eqref{CS} on the torus, over the whole critical space $L^2_+(\T)$ and to use them to obtain a complete and explicit description of the dynamics.

\medskip

To construct the Birkhoff coordinates, we denote by $(f_n)_{n\ge0}$ the orthonormal eigenbasis of $L^2_+(\T)$ furnished by the eigenvectors of $L_u$ (Theorem~\ref{thm:Lax}). First, since the eigenvalues are simple, each $f_n$ is determined only up to a unimodular phase. Second, the amplitude $\langle u\mid f_n\rangle$ vanishes exactly when $\gamma_n=0$ (cf. eqt. \eqref{eq:Xn=0}), so its phase degenerates at closed gaps. We resolve both  through a \emph{transport gauge} (Convention~\ref{conv:transport}), which fixes the phase of the eigenbasis by requiring
\begin{equation}\label{eq:transport-gauge}
    \langle1\mid f_0\rangle>0,\qquad \langle Sf_{n-1}\mid f_n\rangle>0,\quad\text{for } n\ge1,
\end{equation}
where $S$ is the shift $Sf(x)=e^{ix}f(x)$. The point of this choice is that both quantities are nonzero for \emph{every} $u\in L^2_+(\T)$, even when $\gamma_n=0$ there $\langle Sf_{n-1}\mid f_n\rangle=1$ (Lemma~\ref{lem:geometry}). The observable phase is thereby carried by $\langle u\mid f_n\rangle$, which depends continuously on $u$ (see Remark~\ref{rem:why_transport}).

To each $u\in L^2_+(\T),$ we then associate the coordinates
\begin{equation}\label{eq:intro_zeta}
   u\ \longleftrightarrow\ \Phi(u)=(\zeta_n(u))_{n\ge0},
   \qquad
   \zeta_n(u):=\sqrt{\gamma_n(u)}\;e^{\,i\arg\langle u\mid f_n\rangle},
   \qquad \gamma_0:=\lambda_0 .
\end{equation}
A central computation (Corollary~\ref{cor:kappa_L2}) shows that the modulus $|\zeta_n|$ is a function of the spectrum alone: $|\langle u\mid f_n\rangle|^2=\gamma_n\kappa_n$ where  
\begin{equation}\label{eq:intro_kappa}
   \kappa_n(u):=\prod_{p\neq n}\Bigl(1-\frac{\gamma_p(u)}{\lambda_p(u)-\lambda_n(u)}\Bigr),
\end{equation}
is a strictly positive, absolutely convergent infinite product (Corollary~\ref{cor:kappa_L2}, Lemma~\ref{lem:kappa_bound}). Consequently \eqref{eq:intro_zeta} may equivalently be written
\begin{equation}\label{eq:intro_zeta2}
   \zeta_n(u)=\frac{\langle u\mid f_n\rangle}{\sqrt{\kappa_n(u)}}.
\end{equation}
Note that $|\zeta_n(u)|^2=\gamma_n(u)$. Thus only the \emph{phase} of $\zeta_n$ carries information beyond the spectrum. The map $\Phi$ is the \emph{nonlinear Fourier transform}  for \eqref{CS}.

\medskip
Our first main result identifies the Birkhoff map $\Phi$ as a global, norm-preserving change of variables.

\begin{theorem}[Birkhoff map, see Theorems~\ref{thm:homeo}]\label{thm:A}
The map 
$$
    \Phi: L^2_+(\T)\to \ell^2(\N_0), \qquad \Phi(u)=(\zeta_n(u))_{n\ge0}:=\zeta
$$
defined by \eqref{eq:intro_zeta}-\eqref{eq:intro_zeta2}, is a homeomorphism, and it is norm-preserving:
\[
   \|\Phi(u)\|_{\ell^2}=\|u\|_{L^2}\,, \qquad u\in L^2_+(\T).
\]
\end{theorem}

Our second main result shows that $\Phi$ turns \eqref{CS} into a rotation.

\begin{theorem}[Linearization of the flow, see Theorem~\ref{thm:linearize}]\label{thm:B}
Let $u\in C(\R;L^2_+(\T))$ solve \eqref{CS} with $\zeta:=\Phi(u(0))$. Then for every $n\geq 0 ,$ $t\in \R,$
\[
   \zeta_n(u(t))=e^{-i\omega_n t}\,\zeta_n(u(0)),
   \qquad
   \omega_n:=\sum_{j=0}^{n-1}(2\lambda_j+1),\quad \omega_0:=0,
\]
where the frequencies $\omega_n$ depend only on the conserved spectrum $(\lambda_j)$.
\end{theorem}

\begin{remark}[The zeroth frequency]\label{rem:omega0}
The relation $\omega_0=0$ is not an accident of the empty sum: it reflects the fact that the constant potentials $u_0=a\in\C^{*}$ are \emph{stationary} solutions of \eqref{CS}. Indeed, for a spatially constant $u$ one has $\p_x^2u=0$ and $\p_x\Pi(|u|^2)=0$, whence $\p_t u=0$. Accordingly the coordinate $\zeta_0$, of modulus $\sqrt{\lambda_0}$, is a genuine constant of motion and does not rotate. 
\end{remark}

Because $\Phi$ is a norm-preserving homeomorphism (Theorem~\ref{thm:A}) and the actions $|\zeta_n|^2=\gamma_n$ are conserved (Equation~\ref{eq:intro_zeta}), Theorem~\ref{thm:B} yields an explicit solution formula 
$$
    u(t)=\Phi^{-1}\bigl((e^{-i\omega_n t}\zeta_n)_{n\geq0}\bigr)
$$
and a complete qualitative description of the dynamics. For arbitrary $L^2$ data we obtain:

\begin{corollary}[Precompactness and Almost-periodicity, Corollary~\ref{cor:dynamics}]\label{cor:dynamics_intro}
For every $u_0\in L^2_+(\T)$, the orbit $\{u(t):t\in\R\}$ is precompact in $L^2_+(\T)$. Moreover, the map
$t\mapsto u(t)$ is (Bohr) almost periodic.
\end{corollary}

The precompactness was already obtained by the author \cite[Theorem~1.4]{badreddine2024global}. We recover it here, together with almost periodicity, as an immediate consequence of the linearization of Theorem~\ref{thm:B}. 

\medskip

\emph{The Finite-gap manifolds and Quasi-periodicity.}
The construction of $\Phi$ proceeds in two stages and it is convenient to state here the finite-gap result on which everything rests. 

A potential $u\in L^2_+(\T)$ is a \emph{finite-gap potential} if only finitely many gaps are open, i.e.\ there is $N\ge0$ with $\gamma_N(u)>0$ and $\gamma_n(u)=0$ for all $n>N$ \cite{badreddine2024traveling}. In this case the top of the spectrum forms an arithmetic tower generated by a finite Blaschke product \cite[Prop.~5.2]{badreddine2024traveling}: there exist $\lambda_u\in\R$ and $\psi_u$ a Balscke product of degree $N$ with
\[
   L_uS^k\psi_u=(\lambda_u+k)S^k\psi_u\qquad\text{for all }k\ge0,
\]
and by \cite{badreddine2024traveling} such potentials are precisely the rational functions of a prescribed degree. Fixing the minimal such degree $N$ (the \emph{spectral index}) defines the classes
\[
   \U_N:=\{u\ \text{finite-gap potential}:\ \mathcal N(u)=N\},\qquad \U_0:=\C^{*},
\]
and the associated \emph{spaces}
\[
   \Omega_N:=\{\zeta\in\C^{N+1}:\zeta_N\neq0\}.
\]

\begin{theorem}[Finite-gap bijection; Theorem~\ref{thm:main}]\label{thm:bij-intro-UN}
For every $N\ge0$, the Birkhoff map restricts to a bijection $$\Phi:\U_N\to\Omega_N\qtq{ and }\Phi(0)=0.$$ Moreover $\Phi$ maps $\{u\in\U_N:\lambda_0(u)>0\}$ and $\{u\in\U_N:\lambda_0(u)=0\}$ onto $\{\zeta\in\Omega_N:\zeta_0\neq0\}$ and $\{\zeta\in\Omega_N:\zeta_0=0\}$ respectively.
\end{theorem}

\begin{remark}\label{rem:manifold}
Since $\Phi$ is a homeomorphism and $\Omega_N$ is an open subset of $\C^{N+1}$, one deduces that teh set of finite gap potentials $\U_N$ is a topological manifold, homeomorphic to $\Omega_N$, of real dimension $2N+2$. 
\end{remark}

For finite-gap data the rotation of Theorem~\ref{thm:B} takes place on a finite-dimensional torus, and we obtain a sharp description, including a criterion for genuine periodicity.

\begin{corollary}[Quasi-periodicity and periodicity; Corollary~\ref{cor:periodic}]\label{cor:periodic_intro}
Let $u_0\in\U_N$, $\zeta=\Phi(u_0)$, and $\Ks:=\{1\le n\le N:\zeta_n\neq0\}$ with $d:=|\Ks|$.
\begin{enumerate}
   \item[\textup{(i)}] There is a continuous map $F:\T^{d}\to\U_N\subset L^2_+(\T)$ with
   \[
      u(t)=F\bigl((-\omega_n t\ \mathrm{mod}\ 2\pi)_{n\in\Ks}\bigr),\qquad t\in\R.
   \]
   In particular $u$ is quasi-periodic with frequency vector $(\omega_n)_{n\in\Ks}$, and
   $u\equiv u_0$ if $d=0$ (i.e.\ $N=0$).
   \item[\textup{(ii)}] The solution $t\mapsto u(t)$ is periodic if and only if the
   frequencies $(\omega_n)_{n\in\Ks}$ satisfy 
   $\omega_n/\omega_m\in\mathbb Q$ for all $n,m\in\Ks$. In particular, if $\Ks=\{n_0\}$ is a
   singleton, then $u$ is $\tfrac{2\pi}{\omega_{n_0}}$-periodic in time. This is the one-gap
   (traveling-wave) case of \cite{badreddine2024traveling}.
\end{enumerate}
\end{corollary}

\smallskip
{\textbf{Strategy of the Proof.}}
We build $\Phi$ in two stages: first on the finite-gap potentials $\U_N$, where all the spectral quantities are finite and explicit (\S\ref{s:birkhoff}), and then on all of $L^2_+(\T)$ by a limiting procedure (\S\ref{S:5}). In the case of the finite-gap manifolds (\S\ref{s:birkhoff}), the reconstruction of $u$ from its spectral data rests on the Hardy representation \eqref{eq:hardy_rep}: identifying $L^2_+(\T)$ with the Hardy space of the disk, every $u$ is recovered as
\begin{equation}\label{eq:intro_hardy}
   u(z)=\bigl\langle(\Id-zM)^{-1}X\mid Y\bigr\rangle_{\ell^2(\N_0)},\qquad z\in\D,
\end{equation}
from the \emph{spectral triple} $X_n=\langle u\mid f_n\rangle$, $Y_n=\langle1\mid f_n\rangle$ and $M=(\langle f_p\mid Sf_n\rangle)_{n,p}$, the matrix of the backward shift $S^{*}$ in the eigenbasis. The whole difficulty is to reconstruct $(X,Y,M)$ from the coordinates $\zeta$ alone. The key device is the \emph{generating function}
\[
   \beta_u(z)=\bigl\langle(L_u-z)^{-1}u\mid u\bigr\rangle=\sum_{n}\frac{|\langle u\mid f_n\rangle|^2}{\lambda_n-z}.
\]
We prove in Corollary~\ref{cor:kappa} that for and $u\in\U_N$,
\begin{equation}\label{eq:gen_finite}
   1-\beta_u(z)=\prod_{p\in\K}\frac{\lambda_{p-1}+1-z}{\lambda_p-z},
   \qquad\K=\{n:\gamma_n>0\}.
\end{equation}
And, taking residues, the fundamental identity $|\langle u\mid f_n\rangle|^2=\gamma_n\kappa_n$ used in \eqref{eq:intro_zeta2} follows. This expresses each amplitude modulus through the spectrum of $L_u$. Lemma~\ref{lem:reconstruction} then recovers $Y$ and $M$ from $(X,\lambda)$, which gives \emph{injectivity} of $\Phi$ on each $\U_N$ (Theorem~\ref{thm:inj}).
For \emph{surjectivity} we construct the inverse $\Psi$ explicitly in Subsection \S\ref{ss:construction}: from $\zeta\in\Omega_N$ we synthesize a candidate spectrum, amplitudes $X$, a vector $Y$ and a matrix $M$, and define $u=\Psi(\zeta)$ by \eqref{eq:intro_hardy}. The difficulty is to prove that this \emph{synthetic} object is the genuine spectral data of an honest potential in $\U_N$. This is achieved through a set of operator identities (Proposition~\ref{prop:struct}),
\[
   MM^{*}=\Id,\qquad M^{*}M=\Id-YY^{*},\qquad \ker M=\C Y,
\]
which say that the synthesized $M^{*}$ is an isometry with one-dimensional defect. Proposition~\ref{prop:virial} upgrades $\|M\|\le1$ to a \emph{strict} spectral radius $\spr(M_{\le N})<1$, forcing $u=\Psi(u)$ to be rational of degree $\le N$ (Lemma~\ref{lem:rational}). A purity argument via Wold's decomposition (Proposition~\ref{prop:purity}) shows that $M^{*}$ is a pure shift with wandering vector $Y$, producing a unitary $U$ intertwining the synthetic objects $(M^{*},X,Y)$ with the true $(S^{*},u,1)$ (Lemma~\ref{lem:U}). Feeding this into Proposition~\ref{prop:index} identifies the spectral index and yields $\Psi(\zeta)\in\U_N$ with $\Phi(\Psi(\zeta))=\zeta$.

\smallskip
We pass from finite gaps to all of $L^2_+(\T)$ in \S\ref{S:5}. Here, three points must be settled to pass to the limit $N\to\infty$. First, the scalings $\kappa_n$ are finite products only on $\U_N$. In general they must be convergent infinite products, which requires the summability of the gaps. This we obtain from a Weyl-type upper bound $\sum_{n\ge1}\gamma_n\le\|u\|_{L^2}^2$ in Proposition~\ref{prop:weyl}. Passing to the limit in the truncations of \eqref{eq:gen_finite} then yields both the \emph{trace formula}
\[
   \|u\|_{L^2}^2=\lambda_0+\sum_{n\ge1}\gamma_n\qquad\text{for every }u\in L^2_+(\T)
\]
(Proposition~\ref{prop:isometry}), which makes $\Phi$ norm-preserving and $\ell^2$-valued via $|\zeta_n|^2=\gamma_n$, and the infinite-product formula \eqref{eq:intro_kappa} for $\kappa_n$  with the uniform bound $\kappa_n\le e^{\|u\|_{L^2}^2}$ (Corollary~\ref{cor:kappa_L2}, Lemma~\ref{lem:kappa_bound}). Second, upgrading bijectivity of $\Phi$ to bicontinuity requires the eigenbasis to depend continuously on $u$. This is precisely where the transport gauge~\ref{eq:transport-gauge} is essential, and we prove continuity of $u\mapsto \langle u \mid f_n \rangle $ through the norm convergence of the Riesz spectral projectors (Lemma~\ref{lem:eigen_cont}). Third, surjectivity and continuity of the inverse follow from $\Phi$ is a proper map : combining the equicontinuity criterion of \cite{KMV} with the Kolmogorov--Riesz--Fr\'echet theorem, any sequence $(u_k)$ with $\Phi(u_k)$ convergent in $\ell^2$ is precompact in $L^2_+(\T)$.  This is Proposition~\ref{prop:proper}. Continuity, injectivity, properness, and Theorem~\ref{thm:bij-intro-UN} together give the homeomorphism (Theorem~\ref{thm:homeo}).

\smallskip
Once the coordinates are in place, the linearization Theorem~\ref{thm:B} follows by tracking a single phase.

\smallskip
\paragraph{\textbf{Related work and remarks}}
 In the  \emph{defocusing} form, the \eqref{CS} considered throughout this paper arises in the modulation theory of internal waves: it was derived by Pelinovsky \cite{pelinovsky1995intermediate} as the intermediate (and, in the infinite-depth limit) nonlinear Schr\"odinger equation governing the envelope of a quasi-monochromatic wave packet at the interface of two fluid layers, thereby providing a modulation description for the Benjamin--Ono regime, and it was there that the nonlocal structure of the nonlinearity was first identified. An inverse-scattering transform was subsequently developed by Pelinovsky--Grimshaw \cite{pelinovsky1995spectral} and Matsuno \cite{matsuno2004cauchy}. In the focusing regime, the same equation appears as the integrable hydrodynamics of the classical Calogero--Sutherland system and is related to the bidirectional Benjamin--Ono equation, following Abanov--Bettelheim--Wiegmann \cite{abanov2009integrable}. 
The integrability of the intermediate and deep--water equations was recognized early: a spectral, inverse scattering transform was developed by Pelinovsky--Grimshaw \cite{pelinovsky1995spectral} and by Matsuno \cite{matsuno2004cauchy}. A direct scattering theory was recently revisited by \cite{frank2025jost, ruoci2026direct}.

Beyond the periodic setting, the model has been studied intensively on the line \cite{gerardlenzmann,killip2025scaling,KMV}. For the finite-depth intermediate equation, de Moura \cite{de2007well} proved local well--posedness for small initial data in $H^s(\R)$ with $s\geq 1$. This result was subsequently pushed to low regularity by Barros \cite{barros2019local} for small initial data in the critical Besov space $B^{1/2}_{2,1}(\R)$, with further sub-critical well--posedness dynamics investigated in \cite{chapouto2026sub}. A global well-posedness result incorporating a non-vanishing boundary condition at infinity was later established in \cite{akahori2025global}. 

Traveling waves and finite-gap potentials on the torus were classified by the author \cite{badreddine2024traveling}; they constitute the family $\U_N$ studied here. More recently, traveling periodic waves and breathers were constructed by Chen and Pelinovsky \cite{chen2025traveling}. The semiclassical or zero-dispersion limit of the periodic flow was analyzed by \cite{badreddine2024zero}, and the threshold for the growth of high Sobolev norms by  \cite{gerardlenzmann,kim2024construction, hogan-Kowalski, chen2026finite}. In particular, in sharp contrast with the global regularity of the defocusing flow studied here, the focusing equation admits finite-time blow-up. Singular solutions were constructed in \cite{kim2024construction, jeong2024quantized, chen2026finite}: the quantized blow-up rate
\[
   \|u(t)\|_{H^s(\R)}\sim_{s,u_0}\frac{1}{(T-t)^{2s}}\qquad\text{as }t\to T^{-},\quad s>0.
\] 
A classification of single-bubble blow-up solutions was obtained by Jeong--Kim--Kim--Kwon \cite{jeong2026classification}. Finally, soliton resolution for the focusing equation was proved by Kim--Kwon \cite{kimKwon}.

\medskip

\paragraph{\textbf{Organization of the paper}}
Section~\ref{S:2} fixes notation and studies the spectral theory of the Lax operator $L_u$, recalling from \cite{badreddine2024global,badreddine2024traveling} how $u$ interacts with its eigenbasis, introduces the transport gauge, and the class $\U_N$ of finite-gap potentials. Section~\ref{ss:generating} introduces the generating function, the Cauchy interpolation lemma, the Birkhoff scalings $\kappa_n$, and the reconstruction of the Hardy data. Section~\ref{s:birkhoff} proves the bijectivity of $\Phi$ on each $\U_N$: injectivity in \S\ref{ss:Injectivity} and, via the explicit inverse, surjectivity in \S\ref{SS:4}. Section~\ref{S:5} extends $\Phi$ to all of $L^2_+(\T)$ and proves it to be a norm-preserving homeomorphism onto $\ell^2(\N_0)$. Section~\ref{S:6} proves the linearization Theorem~\ref{thm:B} and deduces almost periodicity and precompactness of general orbits, and quasi-periodicity and the periodicity criterion for finite-gap data.

\subsection*{Acknowledgements}
The author is partially funded by the European Research Council, grant agreement ``CritPDEsRand'' no.~101219486. Part of this work was carried out while the author was a Hedrick fellow at the University of California, Los Angeles, supported by an AMS--Simons Travel Grant.

\section{Notation and preliminaries}\label{S:2}

In this paper, we adopt the following  notation. We work in the periodic setting $\mathbb{T} = \mathbb{R} / 2\pi\mathbb{Z} \simeq \mathbb{S}^1$. 
Here and after, for a function $f \in L^2(\mathbb{T})$, its Fourier coefficients are defined by
\begin{equation}\label{FT}
\widehat f(k) = \frac{1}{2\pi} \int_0^{2\pi} e^{-ik x} f(x)\,dx, \quad k \in \mathbb{Z},
\end{equation}
and $f(x) = \sum_{k \in \mathbb{Z}} \widehat f(k) e^{ik x}$ is the associated Fourier expansion. We equip $L^2(\mathbb{T})$ with the normalized inner product, taken to be linear in the first argument:
\begin{equation}
    \langle f \mid g \rangle = \frac{1}{2\pi} \int_0^{2\pi} \overline{g(x)}
    \,f(x)  dx.
\end{equation}

Moreover, $\mathbb N=\{1,2,\dots\}$ denotes the positive integers and $\N_0=\{0,1,2,\dots\}$. We write $(e_n)_{n\ge0}$, $e_n=(\delta_{nk})_k$, for the canonical basis of $\ell^2(\N_0)$, and identify a finitely supported sequence with its zero-extension in $\ell^2(\N_0)$.

The Hardy space $L^2_+(\mathbb{T})$ on the unit circle is the closed subspace of functions with non-negative Fourier frequencies:
\begin{equation}
L^2_+(\mathbb{T}) := \{ f \in L^2(\mathbb{T}) :\, \operatorname{supp}(\widehat{f}\,) \subseteq 
 \N_0 \}.
\end{equation}
We denote the Riesz-Cauchy-Szeg\H{o} projection onto $L^2_+(\mathbb{T})$ by
\begin{equation}
\Pi : L^2(\mathbb{T}) \to L^2_+(\mathbb{T}), \qquad \Pi f(x) = \sum_{k \ge 0} \widehat{f}(k) e^{ikx} 
\end{equation}
 The shift operator $S$ acts on $L^2_+(\mathbb{T})$ by multiplication, $(Sh)(x) = e^{ix}h(x)$. It is an isometry satisfying
\begin{equation}
    S^{*}1 = 0, \qquad SS^{*} = \operatorname{Id} - \langle \,\cdot\, \mid 1 \rangle 1.
\end{equation}
The potential $u$ can then be canonically reconstructed  via the shift operator $S$. Indeed,
\begin{equation}\label{eq:u-shift}
    u(x)=\sum_{k\ge0}\langle u\mid e^{ikx}\rangle \,e^{ikx} = \sum_{k\ge0}\langle u\mid S^k 1\rangle \,e^{ikx}= \sum_{k\ge0}\langle S^{*k}u\mid 1 \rangle \,e^{ikx}
\end{equation}
Throughout, elements $h \in L^2_+(\mathbb{T})$ are identified with holomorphic functions $h(z) = \sum_{k=0}^\infty \widehat{h}(k) z^k$ on the open unit disk $\mathbb{D} = \{z \in \mathbb{C} : |z| < 1\}$.

\smallskip
For a sufficiently smooth potential $u$, the unbounded self-adjoint Lax operator $L_u$ with domain $\operatorname{Dom} L_u= H^1_+(\mathbb{T})$ and the skew-adjoint operator $P_u$\,, acting on $L^2_+(\mathbb{T})$ are defined in \cite{badreddine2024global} by
\begin{equation}\label{Lax op}
    L_u f := -i\partial_x f + u \Pi(\bar{u} f) \quad \text{and} \quad 
    P_u f := -u \Pi(\partial_x \bar{u} f) + {\partial_x u} \Pi({\bar{u}} f) + iu \Pi\big(|u|^2\Pi(\bar u  f)\big).
\end{equation}
The evolution equation \eqref{CS} produces then the Lax equation
\begin{equation}
    \partial_t L_{u(t)} = [P_{u(t)}, L_{u(t)}].  \label{Lax eq}
\end{equation}

In \cite{badreddine2024global}, the rigorous definition of the Lax operator in low-regularity settings has been extended to $u\in L^2_+(\mathbb{T})$. The following theorem summarizes key spectral results.

\begin{theorem}[Lax operator {\cite{badreddine2024global,badreddine2024traveling}}]\label{thm:Lax}
For any $u\in L^2_+(\mathbb T)$, the operator $L_u$ extends, via its quadratic form of domain $H^{1/2}_+(\mathbb T)$, to a self-adjoint operator on $L^2_+(\mathbb T)$, bounded below and with compact resolvent. Its spectrum is discrete and simple, 
\begin{equation}\label{eq:simplicity_eigenv}
   0\le\lambda_0<\lambda_1<\lambda_2<\dots\to+\infty.
\end{equation}
The corresponding orthonormal eigenfunctions $(f_n)_{n \ge 0}$ form a complete orthonormal basis of $L^2_+(\mathbb{T})$
\begin{equation}
    L^2_+(\mathbb{T}) = \overline{\bigoplus_{n \ge 0} \operatorname{span}\{f_n\}}, \qquad L_u f_n = \lambda_n(u) f_n.
\end{equation}
Moreover, the gaps $(\gamma_n(u))$ satisfy
\begin{equation}\label{eq:gamma[n]}
    \gamma_n(u) := \lambda_n(u)-\lambda_{n-1}(u)-1 \ge 0.
\end{equation}
Finally, the map $u\mapsto\lambda_n(u)$ is Lipschitz on compact subsets of $\Lp(\T)$.
\end{theorem}

\begin{lemma}[Norm-resolvent continuity]\label{lem:nrc}
For every $\kappa\ge1$ and $r>0$ there is $C=C(r,\kappa)$ such that, for all $u,v\in L^2_+$, with $\|u\|_{L^2}\leq r,\ \|v\|_{L^2}\leq r$, we have
\begin{equation}\label{eq:nrc}
   \bigl\|(L_u+\kappa)^{-1}-(L_v+\kappa)^{-1}\bigr\|_{L^2\to H^1}\le C\,\|u-v\|_{L^2}.
\end{equation}
Consequently, if $u_j\to u$ in $\Lp$, then $(L_{u_j}-z)^{-1}\to(L_u-z)^{-1}$ in operator norm,
uniformly for $z$ in compact subsets of $\C\setminus[0,\infty)$.
\end{lemma}

\begin{proof}
Inequality \eqref{eq:nrc} is \cite[Lem.~7.2, (7.5)]{KMV}. 

Now, fix $\kappa\ge1$ and write
$R_w(z):=(L_w-z)^{-1}$ for $z\in \C\bk [0,\infty)$.  The first resolvent identity  reads $R_w(z)=R_w(-\kappa)+(z+\kappa)R_w(z)R_w(-\kappa)$, i.e.
\begin{equation}\label{eq:res_kappa_z}
   \bigl(I+(z+\kappa)R_w(z)\bigr)R_w(-\kappa)=R_w(z)
   =R_w(-\kappa)\bigl(I+(z+\kappa)R_w(z)\bigr).
\end{equation}
Using \eqref{eq:res_kappa_z} for $u$ on the left and for $v$ on the right, one checks
\begin{equation}\label{eq:res_two_sided}
   R_u(z)-R_v(z)
   =\bigl(I+(z+\kappa)R_u(z)\bigr)\bigl(R_u(-\kappa)-R_v(-\kappa)\bigr)\bigl(I+(z+\kappa)R_v(z)\bigr).
\end{equation}

Let $K$ be compact in $\C\setminus[0,\infty)$, $d:=\operatorname{dist}(K,[0,\infty))>0$ and $\rho_K:=\sup_{z\in K}|z|$. By Theorem~\ref{thm:Lax}, each $L_w$  is self-adjoint with $\sigma(L_w)\subseteq[0,\infty)$, so for $z\in K$,
\begin{equation} \label{eq:res_bound_L2}
    \|R_w(z)\|_{L^2\to L^2}\le\frac{1}{\operatorname{dist}(z,\sigma(L_w))}\le\frac1d\,,
   \qquad
   \bigl\|I+(z+\kappa)R_w(z)\bigr\|\le 1+\frac{\rho_K+\kappa}{d}=:C_K,
\end{equation}
both uniform in $w$. Combining \eqref{eq:res_two_sided} with \eqref{eq:nrc} and the previous inequality,
\begin{equation}\label{eq:nrc_complex}
   \bigl\|R_u(z)-R_v(z)\bigr\|_{L^2\to L^2}\le C_K^{2}\,C(r,\kappa)\,\|u-v\|_{L^2},
   \qquad z\in K,\ \|u\|_{L^2},\|v\|_{L^2}\le r .
\end{equation}

As a consequence, if $u_j\to u$ in $\Lp$, all lie in some ball $\|\cdot\|_{L^2}\le r$, and \eqref{eq:nrc_complex} gives
$$
    \sup_{z\in K}\|R_{u_j}(z)-R_u(z)\|\le C_K^{2}C(r,\kappa)\|u_j-u\|_{L^2}\to0.
$$ 
Hence, $(z-L_{u_j})^{-1}\to(z-L_u)^{-1}$ in operator norm, uniformly on $K$.
\end{proof}

\subsection{Spectral identities and rigidities}\label{ss:identities}
Our analysis relies on the structural and spectral rigidity of the Lax operator $L_u$. We recall from \cite{badreddine2024traveling} these essential facts in the following proposition.

\begin{proposition}[{\cite[Section 2]{badreddine2024traveling}}]\label{prop:Lax_spectral}
The Lax operator $L_u$ satisfies the commutation relation
\begin{equation}\label{eq:Lax_comm}
   L_uS=SL_u+S+\langle\,\cdot\mid S^{*}u\rangle\,u .
\end{equation}
As a consequence, the fundamental relations hold: for any $n\ge0$ and $k\ge1$,
\begin{align}
   \label{eq:0_rel} \langle1\mid u\rangle\,\langle u \mid  f_n \rangle &=\lambda_n \langle 1 \mid  f_n \rangle ,\\
   \label{eq:L_rel} (\lambda_n-\lambda_{k-1}-1)\langle Sf_{k-1}\mid f_n\rangle
   &=\langle Sf_{k-1}\mid u\rangle\,\langle u \mid  f_n \rangle .
\end{align}
Moreover, for $ n\ge1 $, 
\begin{equation}\label{eq:Xn=0}
   \gamma_n=0\iff Sf_{n-1}\parallel f_n\iff \langle u \mid  f_n \rangle =0 .
\end{equation}
\end{proposition}

\begin{lemma}[Non-degeneracy]\label{lem:geometry}
For every $u\in\Lp(\mathbb T)$, the spectral data exhibit the following rigidities: For $n\ge1$, 
\begin{enumerate}
   \item[\textup{(i)}] $\lambda_0=0\iff\langle1\mid u\rangle=0$. In this case $f_0=1$.
   \item[\textup{(ii)}] $\langle1\mid f_0\rangle\neq0$ for every $u\in\Lp$. Moreover, if $\lambda_0>0$, then $\langle u\mid f_0\rangle\neq0$.
   \item[\textup{(iii)}] $\gamma_n=0\iff\langle Sf_{n-1}\mid u\rangle=0$.
   \item[\textup{(iv)}] If $\gamma_n>0$ then $\langle Sf_{n-1}\mid f_n\rangle\neq0$.
\end{enumerate}
\end{lemma}

\begin{proof}

For (i), recall that $L_u 1 = -i\partial_x 1 + u \Pi(\bar u) = u \langle 1 \mid u \rangle$ since $u\in L^2_+(\T)$. If $\langle 1 \mid u \rangle = 0$, then $L_u 1 = 0$. And since $L_u\geq 0$ by \eqref{eq:simplicity_eigenv}, thus $\la_0=0.$

Conversely, suppose $\lambda_0 = 0$, so $L_u f_0 = 0$ for a normalized eigenfunction $f_0$. Taking the inner product with $f_0$ yields
$$
    0 = \langle L_u f_0 \mid f_0 \rangle = \|f_0\|_{\dot{H}^{1/2}}^2 + \|\Pi(\bar{u} f_0)\|_{L^2}^2.
$$
Because both terms are non-negative, each term must vanish independently. The condition $\|f_0\|_{\dot{H}^{1/2}}^2 = 0$ forces $f_0$ to be a non-zero constant function $f_0 \equiv c \in \mathbb{C} \setminus \{0\}$. Indeed, in Fourier modes, $\|f_0\|_{\dot{H}^{1/2}}^2 = \sum_{k=1}^\infty k|\hat{f}_0(k)|^2 = 0$ forces $\hat{f}_0(k) = 0$ for all $k \ge 1$, so $f_0= c$. Since $f_0$ is normalized as $\|f_0\|_{L^2} = 1$, we have $|c| = 1$.
 The second condition $\|\Pi (\bar u f_0)\|_{L^2}=0$ provides $0 = \Pi(\bar{u} c) = c \, \Pi(\bar{u}) = c \langle 1 | u \rangle$. Since $c \neq 0$, this dictates $\langle 1 | u \rangle = 0$.

For (ii), suppose by contradiction that $\langle 1\mid f_0\rangle=0$. Thus, $f_0=Sh$ for some $h\in\Lp$ with $\|h\|=\|f_0\|=1$. Inserting $f_0=Sh$ into $L_uf_0=\lambda_0f_0$ and using the commutation relation \eqref{eq:Lax_comm},
\[
   \lambda_0\,Sh=L_uSh=SL_uh+Sh+\langle h\mid S^{*}u\rangle\,u .
\]
Applying $S^{*}$ and using $S^{*}S=\Id$ yields
\[
   L_uh=(\lambda_0-1)h-\langle h\mid S^{*}u\rangle\,S^{*}u ,
\]
whence, taking the inner product with $h$,
\[
   \langle L_uh\mid h\rangle=(\lambda_0-1)-|\langle h\mid S^{*}u\rangle|^{2}\le\lambda_0-1<\lambda_0 .
\]
This contradicts the variational lower bound $\langle L_uh\mid h\rangle\ge\lambda_0\|h\|^2=\lambda_0$. Hence $\langle 1\mid f_0\rangle\neq0$, and $\langle u \mid  f_0 \rangle=\lambda_0\langle 1\mid f_0\rangle/\langle 1\mid u\rangle\neq0$ if $\la_0\neq 0$ by \eqref{eq:0_rel} and (i).

For (iii), recall from \eqref{eq:Xn=0} that $\gamma_n = 0 \iff S f_{n-1} \parallel f_n \iff \langle u \mid  f_n \rangle  = 0$. Hence, if $\gamma_n = 0$, then $S f_{n-1} = c f_n$ for some $|c|=1$, and $\langle S f_{n-1} \mid u \rangle = c \langle f_n \mid u \rangle = c \overline{\langle u \mid  f_n \rangle } = 0$. Conversely, suppose $\langle Sf_{n-1}\mid u\rangle=0$ and, for contradiction, $\gamma_n>0$. Thus, by \eqref{eq:L_rel},  $\langle Sf_{n-1}\mid f_m\rangle=0$ for all $m\in\N_0$. As $(f_m)$ is a basis this forces  $Sf_{n-1}=0$, contradicting $\|Sf_{n-1}\|=\|f_{n-1}\|=1$.

For (iv), applying again \eqref{eq:L_rel} with $k = n$ yields 
$$
    \gamma_n \langle S f_{n-1} \mid f_n \rangle = \langle S f_{n-1} \mid u \rangle \langle u \mid  f_n \rangle .
$$ 
As $\gamma_n > 0$  statement (iii) implies $\langle S f_{n-1} \mid u \rangle \neq 0$ and \eqref{eq:Xn=0} $\langle u \mid  f_n \rangle \neq 0$, whence $\langle S f_{n-1} \mid f_n \rangle \neq 0$.
\end{proof}

\begin{convention}[Transport gauge]\label{conv:transport}
Fix $u\in L^2_+(\T)$. As $\sigma(L_u)$ is simple by Theorem~\ref{thm:Lax}, each $f_n$ is unique up to a unimodular factor $ e ^{i\theta}$. We fix this angle using the following normalisation for the $(f_n)_{n\ge0}$ :
\[
   \langle1\mid f_0\rangle>0,\qquad \langle Sf_{n-1}\mid f_n\rangle>0,\ \text{for } n\ge1,
\]
where for $\gamma_n=0$ this reads $f_n:=Sf_{n-1}$ by \eqref{eq:Xn=0}. Both conditions are meaningful by Lemma~\ref{lem:geometry}. This is the gauge in which the eigenbasis depends continuously on $u$; see Lemma~\ref{lem:eigen_cont} and Remark~\ref{rem:why_transport}.

\end{convention}

Henceforth $(f_n)_{n\ge0}$ denotes the eigenbasis of $L_u$ normalised by Convention~\ref{conv:transport}.

\medskip
Given $u\in L^2_+(\T),$ set $M$  the matrix representation of $S^*$ in the orthonormal basis $(f_n)_{n \ge 0}$, and $X, Y \in \ell^2(\mathbb{N}_0)$  the representation vectors of $u$ and $1$, then
\begin{equation}\label{eq:X,Y,M}
    X_n := \langle u \mid f_n\rangle, \qquad Y_n := \langle 1 \mid f_n\rangle, \qquad M_{np} := \langle f_p \mid S f_n\rangle = \langle S^* f_p \mid f_n\rangle.
\end{equation}
The potential $u$ reconstructed from the shift in \eqref{eq:u-shift} reads now via the following  Hardy representation
\begin{equation}\label{eq:hardy_rep}
   u(z)=\bigl\langle(\Id-zM)^{-1}X\mid Y\bigr\rangle_{\ell^2(\N_0)},\qquad z\in\D .
\end{equation}
Indeed,  
$\sum_{k\ge0}(zS^{*})^k=(\Id-zS^{*})^{-1}$ for $|z|<1$.

\subsection{Finite-gap potentials \texorpdfstring{$\U_N$}{U[N]}}\label{ss:fg}
We recall from \cite[\S5]{badreddine2024traveling} the class of potentials studied here. Throughout, $\mathcal B_N$ denotes the set of finite Blaschke products of degree $N$,
\[
   \psi(x)=e^{i\theta}\prod_{k=1}^{N}\frac{e^{ix}-\overline{p_k}}{1-p_k e^{ix}},
   \qquad \theta\in\R,\ p_k\in\D,
\]
with the convention that $\mathcal B_0$ is the set of unimodular constants.

\begin{definition}[Finite-gap potential {\cite[Def.~5.1]{badreddine2024traveling}}]
A potential $u\in\Lp(\mathbb T)$ is a \emph{finite-gap potential} if only finitely many gaps are open, i.e.\ there exists $m\ge1$ such that
\[
   \gamma_n(u)=0\qquad\text{for all }n\ge m.
\]
\end{definition}

By \eqref{eq:Xn=0} a closed gap means $Sf_{n-1}\parallel f_n$. Hence, if $u$ in a finite gap potential then from some rank on the spectrum forms an arithmetic tower generated by a Blaschke product \cite[Prop.~5.2]{badreddine2024traveling}: there exists $\lambda_u\in\R$, $\psi_u\in\mathcal B_N$ for some $N\in\N_0$ with
$$
    L_uS^k\psi_u=(\lambda_u+k)S^k\psi_u, \qquad {\forall k\ge0.}
$$ 
We recall the proof of this property in Proposition~\ref{prop:index}. Following \cite[eq.~(5.4)]{badreddine2024traveling}, we associate with each finite-gap potential its \emph{spectral index}
\begin{equation}\label{eq:N(u)}
    \mathcal N(u):=\min\bigl\{N\in\N_0:\ \exists\,\lambda\in\R,\ \exists\,\psi\in\mathcal B_N,\ 
L_uS^{k}\psi=(\lambda+k)S^{k}\psi\ \ \forall k\in\N_0\bigr\}.
\end{equation}
and we set, for $N\ge1$,
\[
   \U_N:=\{u\ \text{finite-gap potential}:\ \mathcal N(u)=N\}.
\]
The index $\mathcal N(u)$  locates the \emph{last open gap}. That is, for any $u$ a finite-gap potential and $N\ge1$, we have (see also Proposition~\ref{prop:index})
\[
   \mathcal N(u)=N\quad\Longleftrightarrow\quad
   \gamma_N(u)>0\ \text{ and }\ \gamma_n(u)=0\ \ \forall n>N.
\]
And for $N=0$, we set $\U_0:=\C^*$.

\begin{proposition}\label{prop:index}
Let $u\in\Lp(\mathbb T)$ be a finite-gap potential with eigenbasis $(f_n)_{n\ge0}$ normalized via Convention~\ref{conv:transport}. Fix $N\ge1$ and suppose
\begin{equation}\label{eq:star}
   \gamma_N>0\qquad\text{and}\qquad \gamma_n=0\quad\text{for all }n>N.
\end{equation}
Then
\begin{enumerate}
   \item[\textup{(a)}] $f_{N+k}=S^kf_N$ and $L_uS^kf_N=(\lambda_N+k)S^kf_N$ for all $k\ge0$
   \item[\textup{(b)}] $f_N$ is a finite Blaschke product of degree $N$
   \item[\textup{(c)}] $\mathcal N(u)=N$ and $u\in\mathcal{U}_N$. 
\end{enumerate}
\end{proposition}

\begin{proof}
First, let us prove {(a)}. For $n>N$ we have $\gamma_n=0$, so $Sf_{n-1}\parallel f_n$ by \eqref{eq:Xn=0}, and Convention~\ref{conv:transport} sets $f_n=Sf_{n-1}$. Iterating from $n=N+1$ gives $f_{N+k}=S^kf_N$. Since $\gamma_n=0$ for $N<n\le N+k$ we have $\lambda_{N+k}=\lambda_N+k$, whence
$$
    L_uS^kf_N=L_uf_{N+k}=\lambda_{N+k}f_{N+k}=(\lambda_N+k)S^kf_N.
$$

For {(b)}, note by (a), the set $\{S^kf_N:k\ge0\}=\{f_{N+k}:k\ge0\}$ is orthonormal. In particular, for $k\ge1$,
\[
   \widehat{|f_N|^2}(-k)=\langle S^kf_N\mid f_N\rangle=\langle f_{N+k}\mid f_N\rangle=0,
\]
and $|f_N|^2$ being real, all its nonzero Fourier modes vanish, hence $|f_N|^2=\|f_N\|_{L^2}^2=1$ a.e.\ on $\mathbb T$, so $f_N$ is inner. Multiplication by the inner function $f_N$ is an isometry of $\Lp$, so
\[
   f_N\Lp=\overline{\operatorname{span}}\{S^kf_N:k\ge0\}
         =\overline{\operatorname{span}}\{f_{N+k}:k\ge0\},
\]
and therefore its model space \cite[Ch.5]{GMR} is
$$
    K_{f_N}=\Lp\ominus f_N\Lp=\operatorname{span}\{f_0,\dots,f_{N-1}\},
$$ 
of dimension $N$. Since $\dim K_{f_N}=N<\infty$, the inner function $f_N$ is a finite Blaschke product with $\deg f_N=\dim K_{f_N}=N$.

Finally, let us prove (c). By (a)--(b), $f_N\in\mathcal B_N$ generates the tower $L_uS^kf_N=(\lambda_N+k)S^kf_N$, so $\mathcal N(u)\le N$. Conversely, let $\psi\in\mathcal B_{N'}$ satisfy 
$$
    L_uS^k\psi=(\nu+k)S^k\psi \qtq{for all} k\ge0.
$$
Taking $k=0$, $\psi$ is an eigenvector of $L_u$. By simplicity $\psi=c\,f_j$ for some $j\ge0$, $c\in\C^{*}$, and $\nu=\lambda_j$. For every $k\ge0$, $S^k\psi$ is then an eigenvector with eigenvalue $\lambda_j+k$, so $\lambda_j+k\in\{\lambda_n\mid n\in\N_0\}$. Because $(\lambda_n)$ is strictly increasing with $\lambda_n-\lambda_{n-1}\ge1$, an induction gives $\lambda_{j+k}=\lambda_j+k$ for all $k\ge0$, i.e.\ $\gamma_n=0$ for all $n>j$. Since $\gamma_N>0$, this forces $j\ge N$. Applying (b) with $j$ in place of $N$ (valid, as $\gamma_n=0$ for $n>j$) yields $f_j\in\mathcal B_j$, whence 
$$
    N'=\deg\psi=\deg f_j=j\ge N.
$$ 
Therefore $\mathcal N(u)=N$ and $u\in\U_N$.
\end{proof}

\begin{proposition}[Trace formula {\cite[Cor.~5.5]{badreddine2024traveling}}]\label{prop:trace}
Any  finite-gap potential $u\in\U_N$ satisfies $$\|u\|_{L^2}^2=\lambda_0+\sum_{n=1}^N\gamma_n=\lambda_N-N.$$
\end{proposition}
\begin{remark}
    For general $u\in L^2_+,$ the $L^2$--norm is derived in Proposition~\ref{prop:isometry}.
\end{remark}

Finally, \cite{badreddine2024traveling}'s characterization shows that these potentials are exactly the rational functions of prescribed degree.

\begin{theorem}[Characterization {\cite[Thm.~5.4]{badreddine2024traveling}}]\label{thm:char}
Let $N\ge1$. A potential $u$ belongs to $\U_N$ if and only if either $u(x)=Ce^{iNx}$ with $C\in\C^{*}$, or $u$ is the rational function
\begin{equation}\label{eq:rational}
   u(x)=e^{im_0x}\prod_{j=1}^{r}
   \left(\frac{e^{ix}-\overline{p_j}}{1-p_j e^{ix}}\right)^{m_j-1}
   \left(a+\sum_{j=1}^{r}\frac{c_j}{1-p_j e^{ix}}\right),
   \quad p_j\in\D^{*},\ p_k\neq p_j\ (k\neq j),
\end{equation}
where $m_0\in\{0,\dots,N-1\}$, $m_1,\dots,m_r\in\{1,\dots,N\}$ satisfy
$m_0+\sum_{j=1}^{r}m_j=N$, and $(a,c_1,\dots,c_r)\in\C\times\C^{r}$ (with $a\neq0$ if $m_0\neq0$) obey the defocusing constraints
\begin{equation}\label{eq:rational_constraint}
   \overline{a}\,c_j+\sum_{k=1}^{r}\frac{c_j\overline{c_k}}{1-p_j\overline{p_k}}
   =-m_j,\qquad j=1,\dots,r.
\end{equation}
In particular every $u\in\U_N$ is a rational function, smooth on $\mathbb T$, and $\U_N$ is invariant under the flow of \eqref{CS} \cite[Prop.~5.6]{badreddine2024traveling}.
\end{theorem}

\section{Generating function and Cauchy systems }\label{ss:generating}

The candidate Birkhoff coordinates are the rescaled amplitudes $X_n:=\langle u\mid f_n\rangle$. We first record the rescaling factor. For $u\in\Lp$ define the \emph{generating function}
\begin{equation}\label{eq:Hu}
   \beta_u(z):=\bigl\langle(L_u-z)^{-1}u\mid u\bigr\rangle=\sum_{n\ge0}\frac{|X_n|^2}{\lambda_n-z},
   \qquad z\in\C\setminus\{\lambda_n\mid n\in \N_0\},
\end{equation}
absolutely convergent since $\lambda_n\ge n$ and $\sum_n|X_n|^2=\|u\|_{L^2}^2$. We use the conventions 
\begin{equation}\label{eq:gmma-mu[0]}
    \gamma_0(u):=\lambda_0(u),\qquad \mu_0=0
\end{equation}
and for all $n\in\N$, 
\begin{equation}\label{eq:gmma-mu[n]}
    \gamma_n(u)=\lambda_n(u)-\mu_n(u)\ge0, \qquad \mu_n(u):=\lambda_{n-1}(u)+1.
\end{equation}
Iterating  \eqref{eq:gamma[n]} gives for any $n\in \N,$ $m\le n$
\begin{equation}\label{eq:lambda_lower}
    \lambda_n=\lambda_0+n+\sum_{k=1}^{n}\gamma_k,\qquad\text{hence}\qquad 
   \lambda_n-\lambda_m\ge n-m,\qquad \lambda_n\ge n .
\end{equation}

\begin{lemma}[Node identities]\label{lem:node}
For any $u\in\Lp$,
\begin{enumerate}
   \item[\textup{(i)}] If $\lambda_0>0$ then $\beta_u(0)=1$.
   \item[\textup{(ii)}] For every $k\in\N$, if $\gamma_k(u)>0$ then $\beta_u(\mu_k)=1$, .
\end{enumerate}
\end{lemma}

\begin{proof}
Let us start with (i). If $\lambda_0>0$ then $L_u$ invertible. Thus, $L_u1=\langle1\mid u\rangle u$ yields
$$
    L_u^{-1}u=\langle1\mid u\rangle^{-1}1
$$
since  $\langle1\mid u\rangle\neq0$ by Lemma~\ref{lem:geometry}(i). Hence, $$\beta_u(0)=\langle L_u^{-1}u\mid u\rangle=1.$$

Now, assume $\gamma_k(u)>0$ for some $k\in\N$. Thus, by Lemma~\ref{lem:geometry}(iii), we have $\langle Sf_{k-1}\mid u\rangle\neq0$ and $\mu_k\notin\{\lambda_n\mid n\in\N_0\}$. Applying \eqref{eq:L_rel} yields, for any $k\in\N,$
\begin{equation}\label{eq:inj_expand}
    \langle Sf_{k-1}\mid f_n\rangle=\frac{\langle Sf_{k-1}\mid u\rangle X_n}{\lambda_n-\mu_k},\qquad n\in\N_0.
\end{equation}
Expanding $u=\sum_nX_nf_n$ in $\langle Sf_{k-1}\mid u\rangle$ and inserting \eqref{eq:inj_expand} 
\[
   \langle Sf_{k-1}\mid u\rangle
   =\sum_{n}\overline{X_n(u)}\,\langle Sf_{k-1}\mid f_n\rangle
   =\langle Sf_{k-1}\mid u\rangle\sum_{n}\frac{|X_n(u)|^2}{\lambda_n-\mu_k}.
\]
Dividing by $\langle Sf_{k-1}\mid u\rangle\neq0$ yields
\begin{equation}\label{eq:Cauchy_inv_Xp(u)}
    \sum_{n}\frac{|X_n(u)|^2}{\lambda_n-\mu_k}=1.
\end{equation}
That is, $\beta_u(\mu_k)=1$ for every $k\in\N$ with $\gamma_k(u)>0$.
\end{proof}

The inversion of such node conditions is governed by the following lemma.

\begin{lemma}[Cauchy interpolation]\label{lem:cauchy}
Let $d\ge1$ and  let $\nu_1<\Lambda_1<\dots<\nu_d<\Lambda_d$ be real. The linear system
\begin{equation}\label{eq:cauchy_sys}
   \sum_{j=1}^{d}\frac{c_j}{\Lambda_j-\nu_i}=1,\qquad i=1,\dots,d,
\end{equation}
admits a unique solution $(c_1,\dots,c_d)$, given by
\begin{equation}\label{eq:cauchy_res}
   c_k=\frac{\prod_{i=1}^d(\Lambda_k-\nu_i)}{\prod_{j\neq k}(\Lambda_k-\Lambda_j)}\,.
\end{equation}
Every $c_k$ is strictly positive, and satisfies
\begin{equation}\label{eq:cauchy_trace}
   1-\sum_{j=1}^d \frac{c_j}{\Lambda_j-z}=\frac{\prod_{k}(z-\nu_k)}{\prod_j(z-\Lambda_j)} \qtq{and} \sum_{j=1}^{d}c_j=\sum_{j=1}^{d}\Lambda_j-\sum_{k=1}^{d}\nu_k .
\end{equation}
\end{lemma}

\begin{proof}
Consider $R(z)=\sum_{j=1}^{d}\dfrac{c_j}{\Lambda_j-z}$, which vanishes at infinity and has simple poles at the $\Lambda_j$. System \eqref{eq:cauchy_sys} states that $R(\nu_i)=1$, i.e.\ that $1-R$ vanishes at the $d$ nodes $\nu_i$. Writing $1-R=A(z)/B(z)$ over the monic denominator $B(z)=\prod_{j=1}^{d}(z-\Lambda_j)$, the numerator $A$ has degree $d$. As $1-R(z)\to1$ as $z\to\infty$ and $B$ is monic, $A$ is \emph{monic} of degree $d$, and its roots are precisely the nodes $\nu_i$\,. Hence, $1-R$ is completely determined by 
\begin{equation}\label{eq:cauchy_factored}
   1-R(z)=\frac{\prod_{i=1}^{d}(z-\nu_i)}{\prod_{j=1}^{d}(z-\Lambda_j)} .
\end{equation}
This fixes the $c_j$ uniquely: Since $\operatorname{Res}_{z=\Lambda_k}(1-R)=c_k$, taking the residue of \eqref{eq:cauchy_factored} at $z=\Lambda_k$ gives \eqref{eq:cauchy_res}. Matching the $1/z$ coefficients in
\[
   1-R(z)=1+\frac{\sum_j c_j}{z}+O(z^{-2})
   \qquad\text{and}\qquad
   \eqref{eq:cauchy_factored}=1+\frac{\sum_j\Lambda_j-\sum_i\nu_i}{z}+O(z^{-2})
\]
gives the trace identity \eqref{eq:cauchy_trace}.

Finally, we check $c_k>0$. In the numerator of \eqref{eq:cauchy_res}, as $\Lambda_k-\nu_i>0$ for the $k$ indices $i\le k$ and $\Lambda_k-\nu_i<0$
for the $d-k$ indices $i>k$, so the numerator has sign $(-1)^{d-k}$. In the denominator,
$\Lambda_k-\Lambda_j>0$ for $j<k$ and $\Lambda_k-\Lambda_j<0$ for the $d-k$ indices $j>k$, so the
denominator also has sign $(-1)^{d-k}$. Hence $c_k>0$.
\end{proof}

By Lemma~\ref{lem:geometry} and \eqref{eq:Xn=0}, $X_n\neq0\iff\gamma_n>0$. Thus, we define the set of  \emph{active} indices,
\begin{equation}\label{eq:K}
   \boxed{\ \K:=\{n\ge0:\gamma_n>0\},\qquad \Ks:=\K\setminus\{0\},\qquad
   0\in\K\iff\lambda_0>0,\ }
\end{equation}
where we recall that $\gamma_0:=\la_0.$ Enumerating
$\K=\{p_0<p_1<p_2<\cdots\}$, the active nodes $\mu_{p_j}$ and the active poles $\la_{p_j}$ strictly interlace:
\begin{equation}\label{eq:interlace}
   \mu_{p_0}<\lambda_{p_0}<\mu_{p_1}<\lambda_{p_1}<\mu_{p_2}<\lambda_{p_2}<\cdots .
\end{equation}
Indeed, $\mu_{p_i}<\lambda_{p_i}$ since $\lambda_{p_i}-\mu_{p_i}=\gamma_{p_i}>0$ by \eqref{eq:gmma-mu[n]} and \eqref{eq:K}, and $\lambda_{p_i}<\mu_{p_{i+1}}$ since $p_i\le p_{i+1}-1$ and $(\lambda_n)$ is increasing, so that
\[
   \lambda_{p_i}\ \le\ \lambda_{p_{i+1}-1}\ <\ \lambda_{p_{i+1}-1}+1\ =\ \mu_{p_{i+1}} .
\]
In particular the hypothesis of Lemma~\ref{lem:cauchy} is met by the data $(\mu_p,\lambda_p)_{p\in\K}$ whenever $\K$ is finite.

\begin{corollary}[Birkhoff scaling]\label{cor:kappa}
Let $u\in\U_N$. Then $\beta_u$ is rational with simple poles  at the active eigenvalues:
\begin{equation}\label{eq:product}
   \beta_u(z)=\sum_{p\in\K}\frac{|X_p|^2}{\lambda_p-z},\qquad
   1-\beta_u(z)=\prod_{p\in\K}\frac{\mu_p-z}{\lambda_p-z}=\prod_{p\in\K}\Bigl(1-\frac{\gamma_p}{\lambda_p-z}\Bigr)\,.
\end{equation}
Consequently, for every $n\ge0$, 
\begin{equation}\label{eq:kappa}
   |X_n|^2=\gamma_n\kappa_n\qtq{where} \kappa_n:=\prod_{\substack{p\in\K\\ p\neq n}}\frac{\mu_p-\lambda_n}{\lambda_p-\lambda_n}=\prod_{\substack{p\in\K\\ p\neq n}}\bigl(1-\frac{\gamma_p}{\lambda_p-\lambda_n}\bigr).
\end{equation}
Each $\kappa_n>0$ is a function of the spectrum $(\lambda_p)$ alone and so is $|X_n|^2$.
\end{corollary}

\begin{proof}
By \eqref{eq:Xn=0} and Lemma~\ref{lem:geometry}, $X_p\neq0$ precisely for $p\in\K$. Hence, the generating function $\beta_u$ of \eqref{eq:Hu} is rational with simple poles exactly at the active $\lambda_p$. That implies the first identity in \eqref{eq:product}. 

By Lemma~\ref{lem:node}, $1-\beta_u$ vanishes at every active node, i.e.\ $\beta_u(\mu_q)=1$ for $q\in\K$. That is, $(|X_p|^2)_{p\in\K}$ solves the Cauchy system
\begin{equation}\label{eq:cor_system}
   \sum_{p\in\K}\frac{|X_p|^2}{\lambda_p-\mu_q}=1,\qquad q\in\K .
\end{equation}
Recall that the poles $(\lambda_p)_{p\in\K}$ and nodes $(\mu_p)_{p\in\K}$ interlace  and are equal in number by \eqref{eq:interlace}, so Lemma~\ref{lem:cauchy} applies and we have by \eqref{eq:cauchy_trace}: 
$$1-\beta_{u}(z)=\prod_{p\in\K}\frac{\mu_p-z}{\lambda_p-z}$$
which gives the second identity in \eqref{eq:product}. Its residue formula \eqref{eq:cauchy_res} at the pole $\lambda_n$ gives
\[
|X_n|^2=\operatorname{Res}_{z=\lambda_n}(1-\beta_u)=(\lambda_n-\mu_n)\!\!\prod_{\substack{p\in\K\\ p\neq n}}\frac{\mu_p-\lambda_n}{\lambda_p-\lambda_n} =\gamma_n\kappa_n , \qtq{for} n\in\K.
\]
Each factor is positive because interlacing gives $\mu_p-\lambda_n$ and $\lambda_p-\lambda_n$ the same sign, so $\kappa_n>0$. For $n\notin\K$, $\gamma_n=0=X_n$ and both sides of \eqref{eq:kappa} vanish.
\end{proof}

\begin{remark}[Spectral-measure interpretation]\label{rem:herglotz}
$\beta_u(z)=\int \frac{d\rho(\lambda)}{\lambda-z}$ is the Weyl function of $L_u$ relative to the  vector $u$, with spectral measure $d\rho=\sum_n|X_n|^2\delta_{\lambda_n}$ supported on $\{\lambda_p:p\in\K\}$ and $\|\rho\|=\|u\|_{L^2}^2$. Corollary~\ref{cor:kappa} recovers $d\rho$ from the spectrum alone.
\end{remark}

\begin{lemma}[Reconstruction of the Hardy data]\label{lem:reconstruction}
Let $u\in\Lp(\mathbb T)$ with spectrum $(\lambda_n)_{n\ge0}$, active set $\K$ of \eqref{eq:K}, amplitudes $X_n=\langle u\mid f_n\rangle$. Then the vectors $Y=(Y_n)_n$ and $M=(M_{np})_{n,p}$ of \eqref{eq:X,Y,M} are determined by $(X_n)_n$ and $(\lambda_n)_n$ alone, through the following formulas:
\begin{enumerate}
\item[\textup{(i)}] If $\lambda_0=0$, then $Y=e_0$. If $\lambda_0>0$, then
\begin{equation*}
   Y_n=\frac{\langle1\mid u\rangle\,X_n}{\lambda_n},\qquad
   |\langle1\mid u\rangle|=\Bigl(\sum_{p\in\K}\tfrac{|X_p|^2}{\lambda_p^{2}}\Bigr)^{-1/2},\quad
   \arg\langle1\mid u\rangle=-\arg X_0 .
\end{equation*}
\item[\textup{(ii)}] If $\gamma_{n+1}=0$, then $M_{np}=\delta_{p,n+1}$. If $\gamma_{n+1}>0$, then
\begin{equation*}
   M_{np}=\frac{\overline{\langle Sf_n\mid u\rangle}\,\overline{X_p}}{\lambda_p-\lambda_n-1},\qquad
   |\langle Sf_n\mid u\rangle|=\Bigl(\sum_{p\in\K}\tfrac{|X_p|^2}{(\lambda_p-\lambda_n-1)^2}\Bigr)^{-1/2},\quad
   \arg\langle Sf_n\mid u\rangle=-\arg X_{n+1}.
\end{equation*}
\end{enumerate}
\end{lemma}

\begin{proof}
\emph{(i) The vector $Y$.} If $\lambda_0=0$, then $f_0=1$ and $\langle1\mid u\rangle=0$ by Lemma~\ref{lem:geometry}(i), so $Y_0=1$ and, by \eqref{eq:0_rel}, $Y_n=0$ for $n\ge1$, i.e.\ $Y=e_0$.

If $\lambda_0>0$, then $\lambda_n\ge\lambda_0>0$ and \eqref{eq:0_rel} gives $Y_n=\frac{\langle1\mid u\rangle X_n}{\lambda_n}$. The modulus of $\langle1\mid u\rangle$ is fixed by Parseval identity $\sum_n|Y_n|^2=\|1\|^2=1$:
\[
   |\langle 1\mid u\rangle|^2\sum_{n\in\K}\frac{|X_n|^2}{\lambda_n^2}=1,
   \qquad\text{i.e.}\qquad
   |\langle 1\mid u\rangle|=\Bigl(\sum_{n\in\K}\tfrac{|X_n|^2}{\lambda_n^2}\Bigr)^{-1/2},
\]
and its argument by the gauge $Y_0=\langle1\mid f_0\rangle>0$ together with $Y_0=\langle1\mid u\rangle X_0/\lambda_0$, giving
$$
    \arg\langle1\mid u\rangle=-\arg X_0.
$$ 
 This proves (i).

\emph{(ii) The matrix $M$.} Fix a row $n$. If $\gamma_{n+1}=0$, the gauge \eqref{conv:transport} gives $f_{n+1}=Sf_n$, so $M_{np}=\langle f_p\mid f_{n+1}\rangle=\delta_{p,n+1}$. 

If $\gamma_{n+1}>0$, then by Lemma~\ref{lem:geometry}(iii)  we have $\langle Sf_n\mid u\rangle\neq0$, and $\lambda_n+1\notin\{\lambda_p\mid p\geq 0\}$, so \eqref{eq:L_rel} with $k=n+1$ gives
\begin{equation}\label{eq:inj_M}
   M_{np}(u)=\overline{\langle Sf_n(u)\mid f_p(u)\rangle}
   =\frac{\overline{\langle Sf_n(u)\mid u\rangle}\,\overline{X_p}}{\lambda_p-\lambda_n-1},
\end{equation}
which vanishes for $p\notin\K$. Employing Parseval identity and using that  $\|Sf_n(u)\|=\|f_n(u)\|=1$, we infer 
\[
   1=\sum_{p\in\K}|\langle Sf_n(u)\mid f_p(u)\rangle|^2
    =|\langle Sf_n(u)\mid u\rangle|^2\sum_{p\in\K}\frac{|X_p|^2}{(\lambda_p-\lambda_n-1)^2}\,.
\]
Then, the modulus of $\langle Sf_n(u)\mid u\rangle$ is
\[
   |\langle Sf_n\mid u\rangle|=\Bigl(\sum_{p\in\K}\frac{|X_p|^2}{(\lambda_p-\lambda_n-1)^2}\Bigr)^{-1/2},
\]
and its argument follows from the gauge $\langle Sf_n\mid f_{n+1}\rangle>0$ and \eqref{eq:L_rel}. Namely, $\gamma_{n+1}\langle Sf_n\mid f_{n+1}\rangle=\langle Sf_n\mid u\rangle X_{n+1}$, giving
$$
    \arg\langle Sf_n\mid u\rangle=-\arg X_{n+1}.
$$
This proves (ii).
\end{proof}

\section{The Birkhoff map of \texorpdfstring{$\Phi$}{Phi} on \texorpdfstring{$\U_N$}{U[N]}}\label{s:birkhoff}

Given $N\geq 0,$ we define the spectral space
\[
   \Omega_N:=\{\zeta=(\zeta_0,\dots,\zeta_N)\in\C^{N+1}:\zeta_N\neq0\},
\]
identified with a subset of $\ell^2(\N_0)$ by zero-extension. We partition this space as $\Omega_N = \Omega_N^{+} \sqcup \Omega_N^{0}$, where
\begin{equation}\label{eq:Omega}
\Omega_N^{+} = \{\zeta \in \Omega_N : \zeta_0\neq0\},\qquad
\Omega_N^{0} = \{\zeta \in \Omega_N : \zeta_0=0\}.
\end{equation}

\begin{theorem}[Bijectivity of the Birkhoff map]\label{thm:main}
For every $N\ge0$,  the Birkhoff map
\begin{equation}\label{eq:Phi}
   \Phi:\U_N\to\Omega_N,\qquad
   \Phi(u)=(\zeta_n(u))_{0\le n\le N},
\end{equation}
defined by 
\begin{equation}\label{eq:zeta[n]}
    \zeta_n(u):=\frac{\langle u\mid f_n\rangle}{\sqrt{\kappa_n(u)}}=\frac{X_n(u)}{\sqrt{\kappa_n(u)}}=\sqrt{\gamma_n(u)}\,e^{i\arg\langle u\mid f_n\rangle} 
\end{equation}
with 
\begin{equation}\label{eq:kappa[n]}
    \kappa_n(u):=\prod_{\substack{ p\neq n}}\Bigl(1-\frac{\gamma_p(u)}{\lambda_p(u)-\lambda_n(u)}\Bigr). 
\end{equation}
is a bijection. Moreover, $\Phi$ maps $\{u \in \mathcal{U}_N : \lambda_0 > 0\}$ and $\{u \in \mathcal{U}_N : \lambda_0 = 0\}$ bijectively onto $\Omega_N^{+}$ and $\Omega_N^{0}$ respectively. And its inverse is the  reconstruction map $\Psi$ defined in Subsection~\ref{ss:construction}.
\end{theorem}

\begin{remark}\label{Rk:Phi}
    Note that the map $\Phi$ is well-defined. Indeed, Corollary~\ref{cor:kappa} implies that 
$$
    |\zeta_n|^2=\frac{|X_n|^2}{\kappa_n(u)}=\gamma_n(u)\qtq{with} |\zeta_0|^2=\lambda_0.
$$ 
Proposition \ref{prop:Lax_spectral} and Lemma~\ref{lem:geometry}(i)–(ii) yields $\zeta_n=0\iff \gamma_n=0\iff X_n=0$ and by definition of $u\in\U_N$, $\zeta_N\neq0$ since $\gamma_N>0$. Injectivity is proved in the Subsection~\ref{ss:Injectivity} and surjectivity in Subsection~\ref{SS:4}.
\end{remark} 

\subsection{Injectivity of \texorpdfstring{$\Phi$}{Phi} }\label{ss:Injectivity}
To prove that the Birkhoff map $\Phi$ defined in \eqref{eq:Phi} is injective, the argument reconstructs the entire Hardy dataset $(X,Y,M)$ of \eqref{eq:hardy_rep} from the spectral variable $\zeta$.

\begin{theorem}[Injectivity]\label{thm:inj}
For every $N\ge0$ the Birkhoff map $\Phi:\U_N\to\Omega_N$ of \eqref{eq:Phi} is injective.
\end{theorem}

\begin{proof}
Let $u,v\in\U_N$ with $\Phi(u)=\Phi(v)=\zeta$. By the Hardy representation \eqref{eq:hardy_rep} it
suffices to prove
\[
   X_n(u)=X_n(v),\qquad Y_n(u)=Y_n(v),\qquad M_{np}(u)=M_{np}(v)\qquad \text{for all }n,p\ge0,
\]
where we recall from \eqref{eq:X,Y,M}, $X_n(u)=\langle u\mid f_n(u)\rangle$, $Y_n(u)=\langle1\mid f_n(u)\rangle$, $M_{np}(u)=\langle f_p(u)\mid Sf_n(u)\rangle$ and $(f_n(u))$ is the transport-gauge eigenbasis of Convention~\ref{conv:transport}.

By \eqref{eq:zeta[n]} and \eqref{eq:kappa[n]}, $\lambda_0(u)=|\zeta_0|^2=\lambda_0(v)$ and $\gamma_n(u)=|\zeta_n|^2=\gamma_n(v)$ for $1\le n\le N$. As $u,v\in\U_N$, we have $\gamma_n(u)=\gamma_n(v)=0$ for $n>N$. Thus, by \eqref{eq:lambda_lower},
\[
   \lambda_n(u)=|\zeta_0|^2+n+\sum_{k=1}^{n}|\zeta_k|^2=\lambda_n(v)=:\lambda_n\,,\qquad n\ge0,
\]
so the whole spectrum, the active set $\K$ of \eqref{eq:K}, and by Corollary~\ref{cor:kappa} the positive scalars $\kappa_n$, depend only on $\zeta$, so $\kappa_n(u)=\kappa_n(v)=:\kappa_n$.

Second, let us prove $X(u)=X(v)$. Directly from  \eqref{eq:zeta[n]}, $X_n(u)=\zeta_n\sqrt{\kappa_n}=X_n(v)$ for every $n\geq 0$. Thus, $X(u)=X(v)=:X$. In particular, $X_n\neq0$ exactly on $\K$, and $X_0=0$ iff $\lambda_0=0$ by Lemma~\ref{lem:geometry}.
Finally, Lemma~\ref{lem:reconstruction} gives $Y(u)=Y(v)$, $M(u)=M(v)$.

As $\|M\|=\|S^{*}\|\le1$, the Neumann series in \eqref{eq:hardy_rep} converges for $|z|<1$ and yields $u(z)=v(z)$ on $\D$, i.e.\ $u=v$.
\end{proof}

\begin{remark}\label{rem:X_explicit}
By Corollary~\ref{cor:kappa},
\[
   X_n=\langle u\mid f_n\rangle=\zeta_n\sqrt{\kappa_n},\qquad
   \kappa_n=\prod_{\substack{p\neq n}}\frac{\mu_p-\lambda_n}{\lambda_p-\lambda_n},
\]
with $\lambda_n$, $\kappa_n$ explicit functions of $\zeta$ through $\lambda_n=|\zeta_0|^2+n+\sum_{k\le n}|\zeta_k|^2$ and $\mu_p=\lambda_{p-1}+1$. Together with Lemma~\ref{lem:reconstruction}, this expresses the entire Hardy dataset $(X,Y,M)$,  hence $u$, through $\zeta$. This is the inverse map $\Psi$ constructed in Subsection~\ref{SS:4}, so that $\Psi\circ\Phi=\Id_{\U_N}$.
\end{remark}

\subsection{Surjectivity of \texorpdfstring{$\Phi$}{Phi} and the inverse map}\label{SS:4}

Having shown in Theorem~\ref{thm:inj} that $\Phi$ is injective, we now prove surjectivity by constructing an explicit map
\[
   \Psi:\Omega_N\to\U_N \qtq{and verifying that} \Phi\circ\Psi=\Id_{\Omega_N}.
\]
Combined with injectivity, this shows that $\Phi$ is a bijection with inverse $\Psi$, completing the proof of Theorem~\ref{thm:main}. We recall from \eqref{eq:Omega} that
\[
   \Omega_N=\{\zeta=(\zeta_0,\dots,\zeta_N)\in\C^{N+1}:\zeta_N\neq0\}.
\]
 Throughout this section we identify a finitely supported sequence with the element of $\ell^2(\N_0)$ obtained by extending it by zeros.

\subsubsection{Construction of \texorpdfstring{$\Psi$}{Psi}}\label{ss:construction}

Fix $\zeta\in\Omega_N$. We synthesise successively a spectrum, an amplitude vector $X$, a vector $Y$, a matrix $M$, and finally the potential $u=\Psi(\zeta)$ through the Hardy representation \eqref{eq:hardy_rep}. 

\medskip
\emph{\underline{The spectrum.}} Set
\begin{equation}\label{eq:Psi_spectrum}
   \lambda_0:=|\zeta_0|^2,\qquad
   \gamma_n:=|\zeta_n|^2\ (1\le n\le N),\qquad
   \gamma_n:=0\ (n>N),
\end{equation}
and define the candidate eigenvalues
\begin{equation}\label{eq:Psi_eigenvalues}
   \lambda_n:=\lambda_0+n+\sum_{k=1}^{n}\gamma_k,\qquad n\ge0,
\end{equation}
so that $\lambda_n-\lambda_{n-1}-1=\gamma_n\ge0$. Hence $(\lambda_n)_{n\ge0}$ is strictly increasing with $\lambda_n\to+\infty$, and $\lambda_n=\lambda_N+(n-N)$ for $n\ge N$. Write
$$
    \Lambda:=\operatorname{diag}(\lambda_0,\lambda_1,\dots)
$$
for the associated diagonal operator.  Set $\gamma_0:=\lambda_0$, $\mu_p:=\lambda_{p-1}+1$ and $\mu_0=0$. As in \eqref{eq:K}, the active set is
\begin{equation}\label{eq:Psi_K}
   \K:=\{ n\ge 0:\zeta_n\neq0\}=\{n\ge0:\gamma_n>0\},\qquad \Ks:=\K\setminus\{0\},
\end{equation}
so that $0\in\K\iff\zeta_0\neq0\iff\lambda_0>0$ and $\max\K=N$ since $\zeta_N\neq0$. The active poles $(\lambda_p)_{p\in\K}$ and nodes $(\mu_p)_{p\in\K}$ interlace $\lambda_{p-1}<\mu_p<\lambda_p$.

\medskip
\emph{\underline{The amplitude vector $X$.}} Define the strictly positive spectral weights
\begin{equation}\label{eq:Psi_kappa}
   \kappa_n:=\prod_{\substack{p\in\K\\ p\neq n}}\frac{\lambda_{p-1}+1-\lambda_n}{\lambda_p-\lambda_n}>0,
   \qquad n\ge0,
\end{equation}
as in Corollary~\ref{cor:kappa}, and set
\begin{equation}\label{eq:Psi_X}
   X_n:=\zeta_n\sqrt{\kappa_n}\qquad(n\ge0),
\end{equation}
so that $X_n\neq0\iff n\in\K$ (in particular $X_0=0\iff\lambda_0=0$), $\arg X_n=\arg\zeta_n$, and, using $|\zeta_n|^2=\gamma_n$,
\[
   c_n:=|X_n|^2=\gamma_n\kappa_n\qquad(n\ge0).
\]
By \eqref{eq:cauchy_res}, the numbers $(c_p)_{p\in\K}$ are precisely the unique positive solution of the Cauchy system of Lemma~\ref{lem:cauchy} for the interlacing data $(\lambda_p,\mu_p)_{p\in\K}$. Hence the rational function
\begin{equation}\label{eq:Psi_R}
   R(z):=\sum_{p\in\K}\frac{c_p}{\lambda_p-z}=\sum_{p\ge0}\frac{|X_p|^2}{\lambda_p-z}
\end{equation}
satisfies, by \eqref{eq:cauchy_trace},
\begin{equation}\label{eq:Psi_nodes}
   1-R(z)=\prod_{p\in\K}\frac{\mu_p-z}{\lambda_p-z},\qquad
   R(\mu_q)=1\ (q\in\K),\qquad
   \sum_{p\in\K}c_p=\lambda_0+\sum_{k=1}^{N}\gamma_k .
\end{equation}

\medskip
\emph{\underline{The vector $Y$.}} If $\lambda_0>0$ then $0=\mu_0\in\{\mu_p:p\in\K\}$, so by \eqref{eq:Psi_nodes} the solution satisfies
$R(0)=1$. Introduce the scalar
\begin{equation}\label{eq:Psi_b}
   b:=\Big(\sum_{n\in\K}\frac{c_n}{\lambda_n^{2}}\Big)^{-1/2}e^{-i\arg\zeta_0},
\end{equation}
and set
\begin{equation}\label{eq:Psi_Y}
   Y_n:=
   \begin{cases}
       \dfrac{b\,X_n}{\lambda_n} & n\in\K,\\[4pt]
       0 & n\notin\K,
   \end{cases}
   \qquad(\lambda_0>0).
\end{equation}
If $\lambda_0=0$, put $Y:=e_0$ and $b:=0$. In either case a direct computation gives the two relations
\begin{equation}\label{eq:Y_norm}
   \|Y\|_{\ell^2}=1,
\end{equation}
\begin{equation}\label{eq:LambdaY}
   \Lambda Y=b\,X.
\end{equation}
Indeed, for $\lambda_0>0$,
\[
   \|Y\|^2=|b|^2\sum_{n\in\K}\frac{c_n}{\lambda_n^{2}}=1
   \qquad\text{and}\qquad
   (\Lambda Y)_n=\lambda_n Y_n=b\,X_n
\]
by \eqref{eq:Psi_b}--\eqref{eq:Psi_Y}, both sides of \eqref{eq:LambdaY} vanishing off $\K$. For $\lambda_0=0$ they are immediate, since $\|e_0\|=1$ and $\Lambda e_0=\lambda_0 e_0=0=bX$. The scalar $b$ will be identified with $\langle 1\mid u\rangle$ in Lemma~\ref{lem:U}.

\medskip
\emph{\underline{The matrix $M$.}} For every $n$ with $\gamma_{n+1}>0$, set
\begin{equation}\label{eq:Psi_w}
   |w_n|:=\Big(\sum_{p\in\K}\frac{c_p}{(\lambda_p-\lambda_n-1)^{2}}\Big)^{-1/2},
   \qquad \arg w_n:=-\arg\zeta_{n+1},
\end{equation}
and define the matrix $M=(M_{np})_{n,p\ge0}$ by
\begin{equation}\label{eq:Psi_M}
   M_{np}:=
   \begin{cases}
      \dfrac{\overline{w_n}\,\overline{X_p}}{\lambda_p-\lambda_n-1}, & \gamma_{n+1}>0,\\[8pt]
      \delta_{p,n+1}, & \gamma_{n+1}=0.
   \end{cases}
\end{equation}
Each row is finitely supported ($X_p=0$ off $\K$), and the denominators never vanish: when $\gamma_{n+1}>0$, $\lambda_n+1=\mu_{n+1}\in(\lambda_n,\lambda_{n+1})$ is not an eigenvalue. 

\medskip
\emph{\underline{The potential.}} Once $\|M\|\le1$ is established in Corollary~\ref{cor:contraction}, the Neumann series $(\Id-zM)^{-1}=\sum_{k\ge0}z^kM^k$ converges for $z\in\D$, and we define
\begin{equation}\label{eq:Psi_u}
   u:=\Psi(\zeta),\qquad u(z):=\bigl\langle(\Id-zM)^{-1}X\mid Y\bigr\rangle_{\ell^2(\N_0)},
   \quad z\in\D.
\end{equation}
Subsections~\ref{ss:analyticity}--\ref{ss:equivalence} show that this $u$ is indeed in $\U_N$ and $\Phi(u)=\zeta$.

\subsubsection{Structural identities and analyticity}\label{ss:analyticity}

Throughout this section,  $(X,Y,M)$ are as synthesised in \S\ref{ss:construction} from a fixed $\zeta\in\Omega_N$, and $\Lambda=\operatorname{diag}(\lambda_0,\lambda_1,\dots)$. And $R(z)=\sum_{n\in\K}c_n/(\lambda_n-z)$ is the function defined in \eqref{eq:Psi_R} and satisfies \eqref{eq:Psi_nodes}. We first record two cancellations.

\begin{lemma}[Action of $M$ on $X$ and $Y$]\label{lem:MXMY}
For every $n\ge0$,
\[
   (MX)_n=\begin{cases}\overline{w_n}, & \gamma_{n+1}>0,\\ 0, & \gamma_{n+1}=0,\end{cases}
   \qquad\text{and}\qquad MY=0 .
\]
\end{lemma}

\begin{proof}
Fix $n \geq 0$. If $\gamma_{n+1}=0$, then by \eqref{eq:Psi_M} and \eqref{eq:Psi_X}, 
$$
    (MX)_n=\sum_p\delta_{p,n+1}X_p=X_{n+1}=0,
$$ 
and $(MY)_n=Y_{n+1}=0$ by \eqref{eq:Psi_Y}. 

If $\gamma_{n+1}>0$, then using $|X_p|^2=c_p$ on $\K$ and \eqref{eq:Psi_nodes},
\[
   (MX)_n=\sum_{p\in\K}\frac{\overline{w_n}\,\overline{X_p}}{\lambda_p-\lambda_n-1}\,X_p
         =\overline{w_n}\sum_{p\in\K}\frac{c_p}{\lambda_p-\mu_{n+1}}
         =\overline{w_n}\,R(\mu_{n+1})=\overline{w_n}\,.
\]
For $MY$, on this row we distinguish the two cases: if $\lambda_0=0$, then $Y=e_0$ and from \eqref{eq:Psi_M}, $(MY)_n=M_{n0}=\overline{w_n}\,\overline{X_0}/(\lambda_0-\lambda_n-1)=0$ because $X_0=0$. If $\lambda_0>0$, then $Y_p=bX_p/\lambda_p$, and  $\lambda_n+1=\mu_{n+1}$. Hence, using the partial fraction decomposition $\frac{1}{\lambda_p(\lambda_p-\mu_{n+1})}=\frac{1}{\mu_{n+1}}\bigl(\frac{1}{\lambda_p-\mu_{n+1}}-\frac{1}{\lambda_p}\bigr)$ yields
\[
   (MY)_n=\sum_{p\in\K}\frac{\overline{w_n}\,\overline{X_p}}{\lambda_p-\mu_{n+1}}\cdot\frac{bX_p}{\lambda_p}
   =\overline{w_n}\,b\sum_{p\in\K}\frac{c_p}{\lambda_p(\lambda_p-\mu_{n+1})}
   =\frac{\overline{w_n}\,b}{\mu_{n+1}}\bigl(R(\mu_{n+1})-R(0)\bigr)=0,
\]
where we have used \eqref{eq:Psi_nodes}. Hence $MY=0$.
\end{proof}

\begin{proposition}[Structural identities]\label{prop:struct}
The synthesised matrix $M$ of \eqref{eq:Psi_M} is a bounded operator on $\ell^2(\N_0)$ satisfying
\begin{align}
   \label{eq:struct_I}  & MM^{*}=\Id,\\
   \label{eq:struct_II} & M^{*}M=\Id-YY^{*}\quad\text{and}\quad\ker M=\C Y,\\
   \label{eq:struct_C}  & M\Lambda=\Lambda M+M+(MX)X^{*}.
\end{align}
\end{proposition}

\begin{proof}
First, let us prove the boundedness of $M.$  Write $M=D+F$, where $D_{np}=\delta_{p,n+1}$ on rows with $\gamma_{n+1}=0$ and $D_{np}=0$ otherwise, while $F$ carries the rows with $\gamma_{n+1}>0$.  On one hand, 
$$
    \|Dv\|^2=\sum_{n:\gamma_{n+1}=0}|v_{n+1}|^2\le\|v\|^2 \qtq{so} \|D\|\le1.
$$
On the other hand, $F$ carries  the $ |\Ks| $ rows $n$ with $ n+1\in\Ks $, each of $\ell^2$-norm $|w_n|\bigl(\sum_p c_p/(\lambda_p-\lambda_n-1)^2\bigr)^{1/2}=1$ by \eqref{eq:Psi_w}, so $F$ has rank $\le |\K^*|$ and is bounded. Hence $M$ is bounded. 

Second, let us prove \eqref{eq:struct_I}. We compute $(MM^*)_{nn'}=\sum_p M_{np}\overline{M_{n'p}}$, distinguishing cases. If $\gamma_{n+1}=0$ and $\gamma_{n'+1}=0$, then
$$
    (MM^*)_{nn'}=\sum_p\delta_{p,n+1}\delta_{p,n'+1}=\delta_{nn'}.
$$ 
If $\gamma_{n+1}=0$ and $\gamma_{n'+1}>0$ (so $n\neq n'$), then $(MM^*)_{nn'}=\overline{M_{n',n+1}}=0$, because $M_{n',n+1}$ carries the factor $\overline{X_{n+1}}=0$ ($n+1\notin\K$). The symmetric case is the conjugate. Finally, if $\gamma_{n+1}>0$ and $\gamma_{n'+1}>0$, then for  $\lambda_n+1=\mu_{n+1}$ and $\lambda_{n'}+1=\mu_{n'+1}$ we have
\[
   (MM^*)_{nn'}=\overline{w_n}\,w_{n'}\sum_{p\in\K}\frac{|X_p|^2}{(\lambda_p-\mu_{n+1})(\lambda_p-\mu_{n'+1})}
   =\overline{w_n}\,w_{n'}\sum_{p\in\K}\frac{c_p}{(\lambda_p-\mu_{{n+1}})(\lambda_p-\mu_{n'+1})}.
\]
If $n=n'$ this equals by \eqref{eq:Psi_w} to
$$
    (MM^*)_{nn'}=|w_n|^2\sum_p c_p/(\lambda_p-\lambda_n-1)^2=1.
$$
If $n\neq n'$, the partial fraction
$\frac{1}{(\lambda_p-\mu_{n+1})(\lambda_p-\mu_{n'+1})} =\frac{1}{\mu_{n+1}-\mu_{n'+1}}\bigl(\frac{1}{\lambda_p-\mu_{n+1}}-\frac{1}{\lambda_p-\mu_{n'+1}}\bigr)$
together with \eqref{eq:Psi_nodes} gives
\begin{align*}
    (MM^*)_{nn'}
    =\overline{w_n}\,w_{n'}\sum_p\frac{c_p}{(\lambda_p-\mu_{n+1})(\lambda_p-\mu_{n'+1})}
    &=\frac{\overline{w_n}\,w_{n'}}{\mu_{n+1}-\mu_{n'+1}}\bigl(R(\mu_{n+1})-R(\mu_{n'+1})\bigr)\\
    &=\frac{1}{\mu_{n+1}-\mu_{n'+1}}(1-1)=0.
\end{align*}
In all cases $(MM^*)_{nn'}=\delta_{nn'}$, i.e.\ \eqref{eq:struct_I}. Thus $MM^{*}=\Id$.

Third, let us prove \eqref{eq:struct_C}. Entrywise, \eqref{eq:struct_C} reads $(\lambda_p-\lambda_n-1)M_{np}=(MX)_n\overline{X_p}$. If $\gamma_{n+1}>0$, then $(\lambda_p-\lambda_n-1)M_{np}=\overline{w_n}\,\overline{X_p}=(MX)_n\overline{X_p}$ by \eqref{eq:Psi_M} and Lemma~\ref{lem:MXMY}. If $\gamma_{n+1}=0$, then $(\lambda_p-\lambda_n-1)\delta_{p,n+1}$ is nonzero only at $p=n+1$, where it equals $\lambda_{n+1}-\lambda_n-1=\gamma_{n+1}=0$, and $(MX)_n=0$ by Lemma~\ref{lem:MXMY}. Both sides vanish, proving \eqref{eq:struct_C}.

Finally, let us prove \eqref{eq:struct_II}. By \eqref{eq:struct_I} the operator $P:=M^*M$ is self-adjoint and idempotent, $P^2=M^*(MM^*)M=M^*M=P$. Hence $P$ is an orthogonal projection. Since $\langle Px\mid x\rangle=\|Mx\|^2$, we have $\ker P=\ker M$, so $P$ is the orthogonal projection onto $(\ker M)^{\perp}$, i.e.
\begin{equation}\label{eq:MM_proj}
   M^*M=\Id-P_{\ker M}.
\end{equation}
It remains to prove $\ker M=\C Y$. The inclusion $\C Y\subseteq\ker M$ is Lemma~\ref{lem:MXMY} where we have proved $MY=0$.  Conversely, let $x\in\ker M$. Every row with $\gamma_{n+1}=0$ gives $(Mx)_n=x_{n+1}=0$. Hence, $x_k=0$ for all $k\ge1$ where $\gamma_k=0$, so $x$ is supported in $\K\cup \{0\}$. 

Every row with $\gamma_{n+1}>0$, i.e.\ $\lambda_n+1=\mu_{n+1}$, gives, after dividing by $\overline{w_n}\neq0$,
\begin{equation}\label{eq:ker_cauchy}
   \sum_{p\in\K}\frac{\overline{X_p}\,x_p}{\lambda_p-\mu_{q}}=0,\qquad q\in\Ks .
\end{equation}
Set $m:=|\K|$, so that $|\Ks|=m$ if $\lambda_0=0$ and $|\Ks|=m-1$ if $\lambda_0>0$.

If $\lambda_0=0$, then $0\notin\K$, $\K=\Ks$, and $X_0=0$. The system  \eqref{eq:ker_cauchy} is a homogeneous system of $m$ equations in the $m$ unknowns $(\overline{X_p}\,x_p)_{p\in\K}$, with matrix $\bigl[(\lambda_p-\mu_q)^{-1}\bigr]_{q\in\Ks,\,p\in\K}$, an $m\times m$ Cauchy matrix, hence invertible. Therefore $\overline{X_p}\,x_p=0$, and since $X_p\neq0$ we get $x_p=0$ for every $p\in\K$. Hence $x=x_0e_0=x_0Y$.

If $\lambda_0>0$, then $0\in\K$ and \eqref{eq:ker_cauchy} says that $(\overline{X_p}\,x_p)_{p\in\K}$ lies in the kernel of the $(m-1)\times m$ Cauchy matrix $C=\bigl[(\lambda_p-\mu_q)^{-1}\bigr]_{q\in\Ks,\,p\in\K}$. Any $m-1$ of its columns form an invertible square Cauchy matrix, so $\operatorname{rank}C=m-1$ and $\dim\ker C=1$.  But since $MY=0$ by Lemma~\ref{lem:MXMY}, then this implies that $(\overline{X_p}\,Y_p)_{p\in\K}\in\ker C$. This last vector is nonzero since $\overline{X_p}\,Y_p=b\,c_p/\lambda_p\neq0$. Hence, $\ker C=\C\,(\overline{X_p}\,Y_p)_p$, and so $(\overline{X_p}\,x_p)_p=t\,(\overline{X_p}\,Y_p)_p$ for some $t\in\C$. Dividing by $X_p\neq0$ gives $x_p=tY_p$ for $p\in\K$, i.e.\ $x=tY$. In both cases $\ker M=\C Y$, and since $\|Y\|=1$ we have $P_{\ker M}=YY^*$, thus \eqref{eq:MM_proj} yields \eqref{eq:struct_II}.
\end{proof}

\begin{corollary}[Contraction and finite block]\label{cor:contraction}
For every $v\in\ell^2(\N_0)$, 
\begin{equation}\label{eq:contraction}
   \|Mv\|^2=\|v\|^2-|\langle v\mid Y\rangle|^2,\qquad\text{in particular}\quad \|M\|\le1.
\end{equation} 
Moreover the subspace $E_N:=\operatorname{span}\{e_0,\dots,e_N\}$ is invariant under $M$. Writing $M_{\le N}:=M|_{E_N}$ for the compression, one has $X,Y\in E_N$, $M_{\le N}Y=0$, and
\begin{equation}\label{eq:M_bound}
   \sigma(M_{\le N})\subseteq\overline{\D},\qquad \det M_{\le N}=0.
\end{equation} 
\end{corollary}

\begin{proof}
By \eqref{eq:struct_II}, 
$$
    \|Mv\|^2=\langle M^{*}Mv\mid v\rangle=\|v\|^2-\langle YY^{*}v\mid v\rangle=\|v\|^2-|\langle v\mid Y\rangle|^2\le\|v\|^2,
$$
giving \eqref{eq:contraction}. For invariance, let $p\le N$. If $n>N$ then $\gamma_{n+1}=0$, so $M_{np}=\delta_{p,n+1}=0$ because $n+1>N\ge p$. Thus, $Me_p\in E_N$, i.e.\ $E_N$ is $M$-invariant. Since $\K\subseteq\{0,\dots,N\}$ we have $X,Y\in E_N$, and $M_{\le N}Y=MY=0$ by Lemma~\ref{lem:MXMY}. Finally, $\|M_{\le N}\|\le\|M\|\le1$ gives the spectral inclusion $\sigma(M_{\le N})\subseteq\overline{\D}$, and $M_{\le N}Y=0$ with $Y\neq0$ forces $0\in\sigma(M_{\le N})$, i.e.\ $\det M_{\le N}=0$.
\end{proof}

\begin{proposition}\label{prop:virial}
The finite matrix $M_{\le N}$ has no eigenvalue on $\partial\D$.  Consequently its spectral radius satisfies 
$$
    \spr(M_{\le N})<1.
$$
\end{proposition}

\begin{proof}
By contradiction, suppose $M_{\le N}x=\alpha x$ with $x\in E_N$, $\|x\|=1$ and $|\alpha|=1$. Since $E_N$ is invariant under $M$, $Mx=M_{\le N}x=\alpha x$, so $\|Mx\|=\|x\|=1$. By \eqref{eq:contraction} this forces $\langle x\mid Y\rangle=0$.

Now, we evaluate $\langle\Lambda Mx\mid Mx\rangle$ in two ways. On one hand, $Mx=\alpha x$ with $|\alpha|=1$ gives
\begin{equation}\label{eq: 1<LMx|Mx>}
   \langle\Lambda Mx\mid Mx\rangle=|\alpha|^2\langle\Lambda x\mid x\rangle=\langle\Lambda x\mid x\rangle.
\end{equation}
On the other hand, applying \eqref{eq:struct_C} to $x$,
$\Lambda Mx=M\Lambda x-Mx-\langle x\mid X\rangle\,MX$, so
\[
   \langle\Lambda Mx\mid Mx\rangle
   =\langle M\Lambda x\mid Mx\rangle-\|Mx\|^2-\langle x\mid X\rangle\,\langle MX\mid Mx\rangle.
\]
Using \eqref{eq:struct_II}, we infer
\[
   \langle M\Lambda x\mid Mx\rangle=\langle(\Id-YY^{*})\Lambda x\mid x\rangle
   =\langle\Lambda x\mid x\rangle-\langle\Lambda x\mid Y\rangle\langle Y\mid x\rangle
   =\langle\Lambda x\mid x\rangle,
\]
since $\langle Y\mid x\rangle=0$. Likewise,  $\langle MX\mid Mx\rangle=\langle(\Id-YY^{*})X\mid x\rangle=\langle X\mid x\rangle$, again because $\langle Y\mid x\rangle=0$. With $\|Mx\|^2=1$ this yields
\begin{equation}\label{eq: 2<LMx|Mx>}
   \langle\Lambda Mx\mid Mx\rangle=\langle\Lambda x\mid x\rangle-1-|\langle x\mid X\rangle|^2.
\end{equation}
Equating the two identities \eqref{eq: 1<LMx|Mx>} and \eqref{eq: 2<LMx|Mx>} gives $1+|\langle x\mid X\rangle|^2=0$, which is impossible. Hence, no eigenvalue lies on $\partial\D$, and since $\sigma(M_{\le N})\subseteq\overline{\D}$ by \eqref{eq:M_bound}, all eigenvalues lie in $\D$, i.e.\ $\spr(M_{\le N})<1$.
\end{proof}

\begin{remark}
The restriction to the finite block $M_{\le N}$ in Proposition~\ref{prop:virial} is essential: for the full operator $M$ an eigenvector on $\partial\D$ need not lie in $\operatorname{Dom}\Lambda$, so the quantity $\langle\Lambda Mx\mid Mx\rangle$ used in the proof may be infinite. The passage from the finite block back to $M$ is carried out in Proposition~\ref{prop:purity}, where $\spr(M_{\le N})<1$ is upgraded to $M^{k}\to0$ strongly on all of $\ell^2(\N_0)$.
\end{remark}

\begin{lemma}[Analyticity and rational form]\label{lem:rational}
For every $\zeta\in\Omega_N$ the function $u=\Psi(\zeta)$ of \eqref{eq:Psi_u} is holomorphic on a neighbourhood of $\overline{\D}$ and rational of degree at most $N$
\[
   u(z)=\frac{P(z)}{Q(z)},\qquad Q(z)=\det(\Id-zM_{\le N}),\qquad \deg P,\ \deg Q\le N.
\]
In particular $u $ is smooth.
\end{lemma}

\begin{proof}
Since $\|M\|\le1$, for $|z|<1$ the resolvent $(\Id-zM)^{-1}=\sum_{k\ge0}z^kM^k$ converges and $\xi:=(\Id-zM)^{-1}X\in\ell^2$ solves $\xi-zM\xi=X$. For $n\ge N$ we have $\gamma_{n+1}=0$, so $(M\xi)_n=\xi_{n+1}$ and the $n$-th equation reads $\xi_n-z\xi_{n+1}=X_n$. For $n\ge N+1$, $X_n=0$, whence $\xi_n=z\xi_{n+1}$ and, iterating, $\xi_n=z^k\xi_{n+k}$ for all $k\ge0$. As $\xi\in\ell^2$ is bounded and $|z|<1$, letting $k\to\infty$ gives $\xi_n=0$ for every $n\ge N+1$. Thus $\xi\in E_N$. By the $M$-invariance of $E_N$, the truncated vector solves $(\Id-zM_{\le N})\xi_{\le N}=X_{\le N}$.

By Proposition~\ref{prop:virial}, $\spr(M_{\le N})<1$, so $\Id-zM_{\le N}$ is invertible for all $|z|\le1$, indeed for $|z|<\spr(M_{\le N})^{-1}$, and $\xi_{\le N}=(\Id-zM_{\le N})^{-1}X_{\le N}$. Since $Y\in E_N$,
\[
   u(z)=\langle\xi\mid Y\rangle=\bigl\langle(\Id-zM_{\le N})^{-1}X_{\le N}\mid Y_{\le N}\bigr\rangle.
\]
By Cramer's rule this is $P(z)/Q(z)$ with  
$$Q(z)=\det(\Id-zM_{\le N}) \qtq{and} P(z)=\langle\operatorname{adj}(\Id-zM_{\le N})X_{\le N}\mid Y_{\le N}\rangle,$$ 
two polynomials with $\deg(Q)\leq N+1$ and $\deg(P)\le N$. Writing $Q(z)=\prod_{j}(1-p_jz)$ over the eigenvalues $p_j$ of $M_{\le N}$: the eigenvalue $0$ present by \eqref{eq:M_bound} contributes a trivial factor, so $\deg Q\le N$, and the remaining $|p_j|<1$ place all zeros of $Q$ in $\{|z|>1\}$. Hence, $u=P/Q$ is holomorphic on a neighbourhood of $\overline{\D}$, rational of degree $\le N$, and its boundary values define a smooth element of $\Lp(\mathbb T)$.
\end{proof}

\begin{proposition}[Purity]\label{prop:purity}
As $k\to \infty,$  
$$M^k\to 0  \qquad \text{strongly}.$$ 
Equivalently, the isometry $M^*$ is a pure unilateral shift with one-dimensional wandering subspace $\ker M=\C Y$. Consequently, $\{(M^*)^kY\}_{k\ge0}$ is an orthonormal basis of $\ell^2(\N_0)$.
\end{proposition}

\begin{proof}
 $M^*$ is an isometry follows directly from \eqref{eq:struct_I}: $$\|M^*v\|^2=\langle MM^*v\mid v\rangle=\|v\|^2.$$

Next, we prove {$M^k\to0$ strongly.} For $x\in E_N$, $M^kx=M_{\le N}^kx\to0$ by Proposition~\ref{prop:virial}. For a basis vector $e_j$ with $j\ge N+1$, since $\gamma_{n+1}=0$ for $n\ge N$ and $X_j=0$, we have $Me_j=e_{j-1}$, so $M^{\,j-N}e_j=e_N\in E_N$ and then $M^{\,j-N+k}e_j=M_{\le N}^ke_N\to0$. Thus,
$$M^ke_j\to0 \qtq{for every} j,$$
as $\{e_j\}$ is total and $\|M^k\|\le1$, $M^k\to0$ strongly.

Now, we move to $\{(M^*)^kY\}_{k\geq 0}$ is orthonormal. Indeed, for $l\ge k$, using \eqref{eq:struct_I} repeatedly, 
\[
   \langle(M^*)^kY\mid(M^*)^lY\rangle=\langle Y\mid(M^*)^{l-k}Y\rangle=\langle M^{\,l-k}Y\mid Y\rangle
   =\delta_{kl},
\]
since $MY=0$ by Lemma~\ref{lem:MXMY}. Thus,  $M^{\,l-k}Y=0$ for $l>k$, while $\|Y\|=1$.

Finally, we prove that $\{(M^*)^kY\}_{k\geq 0}$ is complete. The isometry $M^*$ has wandering subspace $\ker(M^*)^*=\ker M=\C Y$, which
is one-dimensional. The condition $M^k\to0$ strongly means precisely that $M^*$ has no unitary part,
i.e.
\[\bigcap_{k\ge0}(M^*)^k\ell^2=\{0\}.\]
By Wold's decomposition \cite[\S1.1]{SzNagyFoias}, $M^*$ is therefore a pure shift and has a  1-dimensional wandering subspace spanned by $Y$, that is,
\[
   \ell^2(\N_0)=\bigoplus_{k\ge0}(M^*)^k(\C Y)=\overline{\operatorname{span}}\{(M^*)^kY:k\ge0\}.
\]
Combined with the orthonormality just shown, $\{(M^*)^kY\}_{k\ge0}$ is an orthonormal basis.
\end{proof}

\subsubsection{Spectral equivalence and proof of the main theorem}\label{ss:equivalence}

Throughout, $u=\Psi(\zeta)$ is the potential of \eqref{eq:Psi_u}. By Lemma~\ref{lem:rational} it is rational, hence smooth and bounded. In particular $L_u=-i\p_x+u\Pi(\bar u \cdot)$ is a well-defined self-adjoint operator on $\Lp(\mathbb T)$ with domain $H^1_+(\mathbb T)$ and compact resolvent.

\begin{lemma}[Spectral equivalence]\label{lem:U}
There is a unitary operator $U:\ell^2(\N_0)\to\Lp(\mathbb T)$ such that
\begin{equation}\label{eq:U_intertwine}
   UY=1,\qquad UM^{*}U^{*}=S,\qquad UX=u,\qquad U\Lambda U^{*}=L_u.
\end{equation}
Consequently, the scalar $b$ of \eqref{eq:Psi_b} equals $\langle 1\mid u\rangle$, the Lax operator $L_u$ has simple spectrum 
$$\sigma(L_u)=\{\lambda_n\}_{n\ge0} \qtq{and the vectors} f_n:=Ue_n$$
are the $L_u$-eigenfunctions normalised according to Convention~\ref{conv:transport}, with $\langle u\mid f_n\rangle=X_n$.
\end{lemma}

\begin{proof}
By Proposition~\ref{prop:purity}, $\{(M^{*})^kY\}_{k\ge0}$ is an orthonormal basis of $\ell^2(\N_0)$, while $\{S^k1\}_{k\ge0}=\{e^{ikx}\}_{k\ge0}$ is the canonical orthonormal basis of $\Lp(\mathbb T)$. Hence
\[
   U\big((M^{*})^kY\big):=S^k1\qquad(k\ge0)
\]
extends to a unique unitary $U:\ell^2\to\Lp$. Taking $k=0$ gives $UY=1$, and for every $k$,
\[
   UM^{*}\big((M^{*})^kY\big)=U(M^{*})^{k+1}Y=S^{k+1}1=S\,U\big((M^{*})^kY\big),
\]
so $UM^{*}=SU$, i.e. $UM^{*}U^{*}=S$ and $UMU^{*}=S^{*}$.

For $k\ge0$, using unitarity,
\[
   \widehat{UX}(k)=\langle UX\mid S^k1\rangle=\langle X\mid(M^{*})^kY\rangle=\langle M^kX\mid Y\rangle
   =\widehat u(k),
\]
where the last identity is a consequence of \eqref{eq:hardy_rep}: $u(z)=\langle(\Id-zM)^{-1}X\mid Y\rangle=\sum_{k\ge0}\langle M^kX\mid Y\rangle z^k$. Thus $UX=u$.

 Since $U^{*}1=Y$ and $U^{*}u=X$,
\[
   \langle 1\mid u\rangle=\langle Y\mid X\rangle=\sum_{n\in\K}Y_n\overline{X_n}.
\]
If $\lambda_0>0$, then $Y_n=bX_n/\lambda_n$, whence by \eqref{eq:Psi_nodes}
$$
    \langle Y\mid X\rangle=b\sum_{n\in\K}\dfrac{c_n}{\lambda_n} = b\,R(0) = b.
$$ 
If $\lambda_0=0$, then $Y=e_0$ and $X_0=0$, so $\langle Y\mid X\rangle=0=b$. In either case $b=\langle 1\mid u\rangle$ and \eqref{eq:LambdaY} reads $\Lambda Y=\langle 1\mid u\rangle X$.

We now prove $U\Lambda U^{*}=L_u$, first on 
$$
    \mathcal D:=\operatorname{span}\{S^k1:k\ge0\}
$$ 
and then everywhere. But, first observe that  $\mathcal D$ is contained in $\operatorname{Dom}(U\Lambda U^{*})$. Indeed,  since $Y$ is supported on $\K$, or \ $Y=e_0$, then 
$$
    Y\in \operatorname{Dom}\Lambda:=\{\eta\in \ell^2(\N_0)\mid \sum_n\lambda_n^2|\eta_n|^2<\infty\}.
$$
Since every row of $M$ is finitely supported, so is every vector $M^{*}e_n=(\overline{M_{np}})_p$. Hence, $M^{*}$ maps finitely supported sequences to finitely supported sequences. As $Y$ is finitely supported (it is supported in $\K$, or equals $e_0$), every $(M^{*})^kY$ is finitely supported and therefore lies in $\operatorname{Dom}\Lambda$. Consequently
$\mathcal D\subseteq\operatorname{Dom}(U\Lambda U^{*})$.

We now show $U\Lambda U^{*}=L_u$ on $\mathcal D$. Conjugating \eqref{eq:struct_C} by $U$ and taking the adjoint, then  using $UMU^{*}=S^{*}$ alongside $UX=u$, gives 
\[
   U\Lambda U^{*}S=S\,U\Lambda U^{*}+S+\langle \cdot \mid S^{*}u\rangle u
.
\]
The Lax operator satisfies the identical commutation relation \eqref{eq:Lax_comm}, meaning $(U\Lambda U^{*}-L_u)S=S(U\Lambda U^{*}-L_u)$. Moreover, since $b=\langle1\mid u \rangle$ and \eqref{eq:LambdaY}, we infer
\[
   U\Lambda U^{*}1=U\Lambda Y=U\big(\langle 1\mid u\rangle X\big)=\langle 1\mid u\rangle u=L_u1.
\]
Consequently,
$$
    (U\Lambda U^{*}-L_u)S^k1=S^k(U\Lambda U^{*}-L_u)1=0 \qtq{for all} k\ge0, 
$$
which means $U\Lambda U^{*}=L_u$ on $\mathcal D$.

Finally, because $u\in L^\infty$ by Lemma~\ref{lem:rational}, the operator $u\Pi({\bar u}\cdot)$ is bounded, then $L_u$ is self-adjoint with domain
$\operatorname{Dom}(L_u)=\operatorname{Dom}(-i\partial_x)=H^1_+(\mathbb T)$, and its graph norm is
\[
   \big|\,\|L_u\phi\|-\|{-i\partial_x}\phi\|\,\big|\le\|u\|_{L^\infty}^2\,\|\phi\|,\qquad\forall\phi\in H^1_+.
\]
Because equivalent graph norms determine the same dense subspaces, a core for $-i\partial_x$ is a core for $L_u$. Since $\mathcal D$ is a core for $-i\p_x$, it is also a core for $L_u$. Therefore, 
$$
    L_u=\overline{L_u|_{\mathcal D}}=\overline{U\Lambda U^{*}|_{\mathcal D}}\subseteq U\Lambda U^{*},
$$
and as both operators are self-adjoint, we conclude $L_u=U\Lambda U^{*}$.

As a consequence, we deduce from $U\Lambda U^{*}=L_u$ and $\Lambda e_n=\lambda_n e_n$, the vectors $f_n:=Ue_n$ form an orthonormal basis with 
$$
    L_uf_n=L_u Ue_n = U\Lambda e_n = \la_n Ue_n= \lambda_n f_n.
$$
As $(\lambda_n)$ is strictly increasing, $\sigma(L_u)=\{\lambda_n\}$ is simple. 

Finally, let us check that $(f_n)$ obeys Convention~\ref{conv:transport}. First observe, $\langle u\mid f_n\rangle=\langle X\mid e_n\rangle=X_n$. When $\lambda_0=0$, one has $f_0=Ue_0=UY=1$ so $Y_0=1>0$. If $\lambda_0>0$ then $Y_0=bX_0/\lambda_0$ and $\arg b=-\arg\zeta_0=-\arg X_0$ gives $bX_0>0$, so $Y_0>0$.

Now, for $\gamma_n>0$, the inner product $\langle Sf_{n-1}\mid f_n\rangle=\overline{M_{n-1,n}}=w_{n-1}X_n/\gamma_n$, and
$\arg w_{n-1}=-\arg\zeta_n=-\arg X_n$ gives $w_{n-1}X_n>0$, so 
$$
    \langle Sf_{n-1}\mid f_n\rangle>0.
$$
And whenever $\gamma_n=0$ one has $M^{*}e_{n-1}=e_n$, hence $$f_n=Ue_n=UM^{*}e_{n-1}=S\,Ue_{n-1}=Sf_{n-1}.$$ 
\end{proof}

\begin{proposition}[$\Psi$ is a right inverse of $\Phi$]\label{prop:surj}
For every $\zeta\in\Omega_N$, the potential $u=\Psi(\zeta)$ of \eqref{eq:Psi_u} belongs to $\U_N$ and $\Phi(u)=\zeta$. Moreover,
$$
    \Psi(\Omega_N^{+})\subseteq\{u\in\U_N:\lambda_0>0\} \qtq{and} \Psi(\Omega_N^{0})\subseteq\{u\in\U_N:\lambda_0=0\}.
$$
\end{proposition}

\begin{proof}
By Lemma~\ref{lem:U}, the Lax operator of $u=\Psi(\zeta)$ has simple spectrum $\{\lambda_n\}_{n\ge0}$, precisely the parameters synthesised in \eqref{eq:Psi_eigenvalues}. In particular $\gamma_n(u)=0$ for $n>N$ and $\gamma_N(u)=|\zeta_N|^2>0$, so $u$ is a finite-gap potential whose last open gap is at index $N$. Knowing the full spectrum, however, does not by itself place $u$ in $\U_N$: this requires identifying the spectral index $\mathcal N(u)$ of \eqref{eq:N(u)}, i.e.\ the \emph{minimal} degree of a Blaschke product generating an arithmetic tower of eigenvalues. Therefore, employing Proposition~\ref{prop:index} with the spectrum $\{\la_n\}$ which satisfies hypothesis $\gamma_N>0$ and $\gamma_n=0$ for $n>N$, and eigenbasis of Lemma~\ref{lem:U} gives $u\in\U_N$.

We compute $\Phi(u)$ using the gauge-normalised eigenbasis $f_n=Ue_n$. The eigenbasis $f_n=Ue_n$ obeys Convention~\ref{conv:transport}, so, with $\kappa_n(u)=\kappa_n$ defined in \eqref{eq:Psi_kappa} and $\langle u\mid f_n\rangle=X_n=\zeta_n\sqrt{\kappa_n}$ by Lemma~\ref{lem:U},
\[
   \zeta_n(u)=\frac{\langle u\mid f_n\rangle}{\sqrt{\kappa_n(u)}}=\frac{\zeta_n\sqrt{\kappa_n}}{\sqrt{\kappa_n}}=\zeta_n,
   \qquad 0\le n\le N .
\]
Hence $\Phi(u)=\zeta$.

Finally, observe by construction $\lambda_0=|\zeta_0|^2$, so $\lambda_0>0\iff\zeta_0\neq0$ and $\lambda_0=0\iff\zeta_0=0$, giving the inclusions.
\end{proof}

\begin{proof}[Proof of Theorem~\ref{thm:main}]
By Theorem~\ref{thm:inj}, $\Phi:\U_N\to\Omega_N$ is injective. By Proposition~\ref{prop:surj}, $\Phi\circ\Psi=\Id_{\Omega_N}$. In particular $\Phi$ is surjective, hence bijective, with $\Phi^{-1}=\Psi$ (so also $\Psi\circ\Phi=\Id_{\U_N}$).  Finally, Proposition~\ref{prop:surj} together with Remark~\ref{Rk:Phi} shows that $\Phi$ restricts to bijections $\{u\in\U_N:\lambda_0>0\}\to\Omega_N^{+}$ and $\{u\in\U_N:\lambda_0=0\}\to\Omega_N^{0}$. 
\end{proof}

\section{Extension of the Birkhoff map to \texorpdfstring{$\Lp(\T)$}{L2+}}\label{S:5}

So far the Birkhoff map $\Phi$ has been defined only on the finite-gap manifolds $\U_N$ Theorem~\ref{thm:main}), where the generating function $\beta_u$ is rational, the active set $\K$ is finite, and every scaling $\kappa_n$ is a finite product (Corollary~\ref{cor:kappa}). The purpose of this section is to extend $\Phi$ to all of $\Lp(\T)$ and to prove that the extension is a norm-preserving homeomorphism onto $\ell^2(\N_0)$ (Theorem~\ref{thm:homeo}).

Three points must be settled. First, the coordinates $\zeta_n=X_n/\sqrt{\kappa_n}$ require $\kappa_n$ to be a convergent, strictly positive infinite product. This rests on the summability $\sum_{n\ge1}\gamma_n<\infty$ which we need to prove. Second, the identity $|\zeta_n|^2=\gamma_n$ must persist, so that $\Phi$ takes values in $\ell^2$ and preserves the norm. This is the trace formula $\|u\|_{L^2}^2=\lambda_0+\sum_{n\ge1}\gamma_n$, valid for every $u\in\Lp$. Third, we also  upgrage bijectivity to bicontinuity, which requires the eigenbasis to depend continuously on $u$ and this is precisely where the choice of the transport gauge of Convention~\ref{conv:transport} is important.

Accordingly, \S\ref{ss:weyl} establishes the Weyl-type bound $\sum_n\gamma_n\le\|u\|_{L^2}^2$ and introduces the truncations $\beta_u^{[m]}$ of $\beta_u$, to which the Cauchy interpolation Lemma~\ref{lem:cauchy} does apply. Letting $m\to\infty$ in \S\ref{ss:trace_kappa} yields the trace formula and the infinite-product formula for $\kappa_n$. \S\ref{ss:cont_Phi} defines $\Phi$ on $\Lp$ and proves it continuous. \S\ref{ss:proper} proves it proper and injective, and concludes that it is a homeomorphism.

\subsection{Weyl bound and truncations}\label{ss:weyl}

Throughout, $u\in\Lp(\T)$ is arbitrary, with spectral data $\lambda_n=\lambda_n(u)$, $\gamma_n=\gamma_n(u)$, $\mu_n=\mu_n(u)$ as in \eqref{eq:gmma-mu[0]}--\eqref{eq:gmma-mu[n]}, and $X_n=\langle u\mid f_n\rangle$, where $(f_n)_{n\ge0}$ is the eigenbasis of Theorem~\ref{thm:Lax} normalised by Convention~\ref{conv:transport}.

The finite-gap scalings $(\kappa_n)$ of Corollary~\ref{cor:kappa} were obtained by applying Lemma~\ref{lem:cauchy} to the \emph{rational} function $\beta_u$. For general $u\in\Lp$ the active set $\K$ may be infinite and $\beta_u$ of \eqref{eq:Hu} need not be rational, so Lemma~\ref{lem:cauchy} does not apply directly. The remedy is to truncate: each truncation is rational, Lemma~\ref{lem:cauchy} applies to it, and the summability $\sum_{n\ge1}\gamma_n<\infty$ (proved  in the upcoming proposition) is what allows the limit $m\to\infty$ to be taken.

\begin{proposition}[Weyl-type upper bound]\label{prop:weyl}
For every $u\in\Lp(\T)$, the sequence $(\lambda_n-n)_{n\ge0}$ is nondecreasing and
\begin{equation}\label{eq:weyl}
   \lim_{n\to\infty}(\lambda_n-n)=\lambda_0+\sum_{n\ge1}\gamma_n\le\|u\|_{L^2}^2 .
\end{equation}
In particular $\sum_{n\ge1}\gamma_n<\infty$.
\end{proposition}

\begin{proof}
By \eqref{eq:lambda_lower}, we obtain $\lambda_n-n=\lambda_0+\sum_{k=1}^n\gamma_k$. Hence, $(\lambda_n-n)$ is nondecreasing with limit $T:=\lambda_0+\sum_{k\ge1}\gamma_k\in[0,\infty]$. Now, let us prove $T<\infty.$

By definition of $L_u=-i\p_x+u\Pi(\bar u \cdot)$, we have for $j\ge0$,
\[
   \langle L_ue^{ijx}\mid e^{ijx}\rangle=\|e^{ijx}\|_{\dot H^{1/2}}^2+\|\Pi(\bar u\,e^{ijx})\|_{L^2}^2
   =j+\sum_{k=0}^{j}|\widehat u(k)|^2=:j+s_j.
\]
By Parseval, $s_j\nearrow\|u\|_{L^2}^2$ as $j\to \infty$. Since $L_u$ is bounded below with compact resolvent, Ky Fan's minimum principle \cite[Thm.~XIII.1]{ReedSimonIV} gives
$$
    \sum_{n=0}^N\lambda_n\le\sum_{j=0}^N\langle L_ue^{ijx}\mid e^{ijx}\rangle=\sum_{j=0}^N(j+s_j),
$$
i.e.
$\sum_{n=0}^N(\lambda_n-n)\le\sum_{j=0}^N s_j$. Dividing by $N+1$ and letting $N\to\infty$: the right side is the Ces\`aro mean of $s_j\to\|u\|^2$, hence $\tfrac{1}{N+1}\sum_{j=0}^N s_j\to\|u\|^2$ as $N\to \infty$, while the left side is a Ces\`aro mean of $(\lambda_n-n)$ which tends to $T$. Thus $T\le\|u\|_{L^2}^2<\infty$.
\end{proof}

\begin{remark}[Reduction to $\K$ infinite]\label{rem:reduction}
The results of \S\ref{ss:weyl}--\S\ref{ss:trace_kappa} are stated for every $u\in\Lp(\T)$, but it suffices to prove them when the active set $\K=\{n\ge0:\gamma_n>0\}$ of \eqref{eq:K} is infinite. Indeed, if $\K=\emptyset$ then $\lambda_0+\sum_{n\ge1}\gamma_n=0$, so $u=0$ by Proposition~\ref{prop:trace} and every assertion is trivial. If $\K$ is finite and nonempty, set $N:=\max\K$. Then $\gamma_N>0$ and $\gamma_n=0$ for $n>N$, so $u\in\U_N$, and the trace formula and the Birkhoff scalings are those already obtained in Proposition~\ref{prop:trace} and Corollary~\ref{cor:kappa}.

We therefore assume from here until the end of \S\ref{ss:trace_kappa} that $\K$ is \emph{infinite}, and we enumerate it as $\K=\{k_0<k_1<k_2<\cdots\}$.
\end{remark}

For $m\ge0$, let
\[
   \beta_u^{[m]}(z):=\sum_{j=0}^{m}\frac{|X_{k_j}|^2}{\lambda_{k_j}-z}
\]
be the \emph{$m$-th truncation} of $\beta_u$, and set
\[
   I_0:=(-\infty,\lambda_{k_0}),\qquad I_i:=(\lambda_{k_{i-1}},\lambda_{k_i})\quad(i\ge1).
\]
Since $\sum_{p\ge0}|X_p|^2=\|u\|_{L^2}^2<\infty$ by Parseval's identity, for every compact $Q\subset\R\setminus\{\lambda_{k_j}\mid j\ge0\}$ with $\delta:=\operatorname{dist}(Q,\{\lambda_{k_j}\})>0$ one has
\[
   \sup_{z\in Q}\bigl|\beta_u(z)-\beta_u^{[m]}(z)\bigr|\le\delta^{-1}\sum_{j>m}|X_{k_j}|^2
   \xrightarrow[m\to\infty]{}0 ,
\]
so $\beta_u^{[m]}\to\beta_u$ locally uniformly off $\{\lambda_{k_j}\}$.

\begin{lemma}\label{lem:trunc}
For every $m\ge0$ and every $0\le i\le m$, the function $\beta_u^{[m]}$ is strictly increasing on $I_i$ and attains the value $1$ at a unique node $\nu_i^{[m]}\in I_i$. These nodes interlace the poles,
\begin{equation}\label{eq:trunc_interlace}
   \mu_{k_i}<\nu_i^{[m]}<\lambda_{k_i},
\end{equation}
and we have
\begin{equation}\label{eq:trunc_fac}
   1-\beta_u^{[m]}(z)=\prod_{i=0}^{m}\frac{z-\nu_i^{[m]}}{z-\lambda_{k_i}},
   \qquad
   \sum_{i=0}^{m}|X_{k_i}|^2=\sum_{i=0}^{m}\bigl(\lambda_{k_i}-\nu_i^{[m]}\bigr).
\end{equation}
Moreover, for each fixed $i$, the sequence $(\nu_i^{[m]})_{m\ge i}$ decreases to $\mu_{k_i}$ as
$m\to\infty$.
\end{lemma}

\begin{proof}On $\R\setminus\{\lambda_{k_0},\dots,\lambda_{k_m}\}$, $\beta_u^{[m]}$ is differentiable with
\[
   \bigl(\beta_u^{[m]}\bigr)'(z)=\sum_{j=0}^{m}\frac{|X_{k_j}|^2}{(\lambda_{k_j}-z)^{2}}>0
\]
 since $k_j\in\K$, so $\beta_u^{[m]}$ is strictly increasing on each $I_i$. We compute its one-sided limits at the endpoints.

For $1\le i\le m$, the endpoints of $I_i=(\lambda_{k_{i-1}},\lambda_{k_i})$ are the poles $\lambda_{k_{i-1}}$ and $\lambda_{k_i}$. Near $\lambda_{k_{i-1}}$ only the $j=i-1$ term is singular so $\beta_u^{[m]}(z)\to-\infty$ as $z\to\lambda_{k_{i-1}}^{+}$. Near $\lambda_{k_i}$ only the $j=i$ term is singular, so $\beta_u^{[m]}(z)\to+\infty$ as $z\to\lambda_{k_i}^{-}$. Being a strictly increasing continuous function with these limits, $\beta_u^{[m]}$ maps $I_i$ bijectively onto $\R$. Thus, there is a unique $\nu_i^{[m]}\in I_i$ with $\beta_u^{[m]}(\nu_i^{[m]})=1$.

For $i=0$ and $z\in I_0=(-\infty,\lambda_{k_0})$ one has $\lambda_{k_j}-z>0$ for all $j$, so each term is positive and $\beta_u^{[m]}(z)>0$. Moreover, $\beta_u^{[m]}(z)\to0^{+}$ as $z\to-\infty$ and $\beta_u^{[m]}(z)\to+\infty$ as $z\to\lambda_{k_0}^{-}$. Hence, $\beta_u^{[m]}$ maps $I_0$ bijectively onto $(0,+\infty)$ and again there is a unique $\nu_0^{[m]}\in I_0$ with $\beta_u^{[m]}(\nu_0^{[m]})=1$.

Over the monic denominator $B(z):=\prod_{i=0}^{m}(z-\lambda_{k_i})$ of degree $m+1$ we have $\beta_u^{[m]}=P/B$ with $\deg P\le m$. Therefore,
\[
   1-\beta_u^{[m]}(z)=\frac{A(z)}{B(z)},
\]
where $A$ is monic of degree $m+1$. The $m+1$ points $\nu_0^{[m]},\dots,\nu_m^{[m]}$ constructed above are distinct real roots of $A$, thus 
\[
   A(z)=\prod_{i=0}^{m}\bigl(z-\nu_i^{[m]}\bigr),\qquad\text{i.e.}\qquad
   1-\beta_u^{[m]}(z)=\prod_{i=0}^{m}\frac{z-\nu_i^{[m]}}{z-\lambda_{k_i}},
\]
which is the first identity in \eqref{eq:trunc_fac}. Comparing the $z^{-1}$ coefficients as
$z\to\infty$,
\[
   1-\beta_u^{[m]}(z)=1+\frac{\sum_i|X_{k_i}|^2}{z}+O(z^{-2}),\qquad
   \prod_{i}\frac{z-\nu_i^{[m]}}{z-\lambda_{k_i}}=1+\frac{\sum_i\lambda_{k_i}-\sum_i\nu_i^{[m]}}{z}+O(z^{-2}),
\]
gives the trace identity in \eqref{eq:trunc_fac}.

Now, to prove \eqref{eq:trunc_interlace}, observe the upper bound is $\sup I_i=\lambda_{k_i}$. For the lower bound, first note $\mu_{k_i}\in I_i$, indeed $\mu_{k_i}<\lambda_{k_i}$ as $\gamma_{k_i}>0$ and $\mu_{k_i}=\lambda_{k_i-1}+1>\lambda_{k_i-1}\ge\lambda_{k_{i-1}}$ as $k_{i-1}\le k_i-1$. Next, for $z\in I_i$ with $i\le m$ every $j>m$ has $\lambda_{k_j}>\lambda_{k_m}\ge\lambda_{k_i}>z$, so
\[
   \beta_u(z)-\beta_u^{[m]}(z)=\sum_{j>m}\frac{|X_{k_j}|^2}{\lambda_{k_j}-z}>0,
\]
the sum being nonempty and positive because $\K$ is infinite. Evaluating at $z=\mu_{k_i}$ and using $\beta_u(\mu_{k_i})=1$ from Lemma~\ref{lem:node} gives $\beta_u^{[m]}(\mu_{k_i})<1=\beta_u^{[m]}(\nu_i^{[m]})$. Strict monotonicity forces $\mu_{k_i}<\nu_i^{[m]}$.

 Fix $i$. For $m\ge i$ and $z\in I_i$,
\[
   \beta_u^{[m+1]}(z)-\beta_u^{[m]}(z)=\frac{|X_{k_{m+1}}|^2}{\lambda_{k_{m+1}}-z}>0,
\]
since $\lambda_{k_{m+1}}>\lambda_{k_m}\ge\lambda_{k_i}>z$. At $z=\nu_i^{[m]}$ this gives $\beta_u^{[m+1]}(\nu_i^{[m]})>\beta_u^{[m]}(\nu_i^{[m]})=1=\beta_u^{[m+1]}(\nu_i^{[m+1]})$, so monotonicity of $\beta_u^{[m+1]}$ yields $\nu_i^{[m+1]}<\nu_i^{[m]}$. Hence $(\nu_i^{[m]})_{m\ge i}$ decreases and, being bounded below by $\mu_{k_i}$, converges to some $\ell\in[\mu_{k_i},\nu_i^{[i]}]$, a compact subset of $I_i$. There $\beta_u^{[m]}\to\beta_u$ uniformly, because for such $z$ and $j>m$ one has $\lambda_{k_j}-z\ge\lambda_{k_j}-\lambda_{k_i}\ge1$ by \eqref{eq:lambda_lower}, whence $0\le\beta_u(z)-\beta_u^{[m]}(z)\le\sum_{j>m}|X_{k_j}|^2\to0$. Therefore
\[
   1=\beta_u^{[m]}\bigl(\nu_i^{[m]}\bigr)\xrightarrow[m\to\infty]{}\beta_u(\ell).
\]
The same argument applies to $\beta_u$ itself: on $I_i$ the series $\beta_u(z)=\sum_{p\in\K}|X_p|^2/(\lambda_p-z)$ has no pole, and on any compact subset of $I_i,$ the function $\beta_u(z)$ is differentiable and  $\beta_u'(z)=\sum_{p\in\K}|X_p|^2/(\lambda_p-z)^2$ converges uniformly, since $|\lambda_p-z|$ is bounded below there and $\sum_p|X_p|^2<\infty$. Hence, $\beta_u$ is $C^1$ and strictly increasing on $I_i$, so $\mu_{k_i}$ is the \emph{unique} solution of $\beta_u=1$ in $I_i$ (Lemma~\ref{lem:node}). As $\beta_u$ is strictly increasing on $I_i$ with unique root $\mu_{k_i}$ of $\beta_u=1$, we conclude $\ell=\mu_{k_i}$, i.e.\ $\nu_i^{[m]}\downarrow\mu_{k_i}$.
\end{proof}

\subsection{Trace formula and Birkhoff scaling}\label{ss:trace_kappa}
We now let $m\to\infty$ in Lemma~\ref{lem:trunc}. Two consequences follow. The trace identity in \eqref{eq:trunc_fac}, combined with the interlacing \eqref{eq:trunc_interlace}, yields the trace formula $\|u\|_{L^2}^2=\lambda_0+\sum_{n\ge1}\gamma_n$ for \emph{every} $u\in\Lp$ (Proposition~\ref{prop:isometry}), extending Proposition~\ref{prop:trace} beyond the finite-gap class. The factorisation in \eqref{eq:trunc_fac}, together with the convergence $\nu_i^{[m]}\downarrow\mu_{k_i}$, yields the identity $|X_n|^2=\gamma_n\kappa_n$ with $\kappa_n$ given by an absolutely convergent infinite product (Corollary~\ref{cor:kappa_L2}). These two results are essential to define the Birkhoff map on all of $\Lp(\T)$ in \S\ref{ss:cont_Phi}.

\begin{proposition}[Trace formula]\label{prop:isometry}
For every $u\in\Lp$,
\[
   \|u\|_{L^2}^2=\lambda_0+\sum_{n\ge1}\gamma_n .
\]
\end{proposition}

\begin{proof}
By Remark~\ref{rem:reduction} we may assume $\K$ infinite. By Lemma~\ref{lem:trunc}, for every $m\ge0$,
\[
   \sum_{i=0}^{m}|X_{k_i}|^2=\sum_{i=0}^{m}\bigl(\lambda_{k_i}-\nu_i^{[m]}\bigr),
\]
and, since $\mu_{k_i}<\nu_i^{[m]}<\lambda_{k_i}$, each summand obeys $0<\lambda_{k_i}-\nu_i^{[m]}<\lambda_{k_i}-\mu_{k_i}=\gamma_{k_i}$. Hence
\[
   \sum_{i=0}^{m}|X_{k_i}|^2<\sum_{i=0}^{m}\gamma_{k_i}\le\sum_{p\ge0}\gamma_p=:T .
\]
Letting $m\to\infty$, the left side increases to $\sum_{p\ge0}|X_p|^2=\|u\|_{L^2}^2$ by Parseval identity (inactive $p$ give $|X_p|^2=0$). Thus, $\|u\|_{L^2}^2\le T$. The reverse inequality $T\le\|u\|_{L^2}^2$ is Proposition~\ref{prop:weyl}. Therefore $\|u\|_{L^2}^2=T=\lambda_0+\sum_{n\ge1}\gamma_n$.
\end{proof}

\begin{corollary}[Birkhoff scaling]\label{cor:kappa_L2}
For every $u\in\Lp$ and $n\ge0$,
\begin{equation}\label{eq:kappa_L2}
   |X_n|^2=\gamma_n\,\kappa_n,\qquad
   \kappa_n:=\prod_{p\neq n}\Bigl(1-\frac{\gamma_p}{\lambda_p-\lambda_n}\Bigr)
            =\prod_{p\neq n}\frac{\mu_p-\lambda_n}{\lambda_p-\lambda_n}\in(0,\infty),
\end{equation}
the product converging absolutely. Each $\kappa_n$ depends only on the spectrum $(\lambda_p)$, and $u\mapsto\kappa_n(u)$ is continuous on $\Lp$.
\end{corollary}

\begin{proof}
  For inactive $n\notin\K$, $\gamma_n=0=X_n$ and both sides of \eqref{eq:kappa_L2} vanish. Now, fix an active index $n=k_{j_0}\in\K$.  For $m\ge j_0$, the pole $\lambda_n$ of $\beta_u^{[m]}$ is simple, whence taking the residue in \eqref{eq:trunc_fac},
\begin{equation}\label{eq:kappa_L2_fixedm}
   |X_n|^2=\operatorname{Res}_{z=\lambda_n}(1-\beta_u^{[m]})
   =\bigl(\lambda_n-\nu_{j_0}^{[m]}\bigr)
   \prod_{\substack{i=0\\ i\neq j_0}}^{m}\frac{\lambda_n-\nu_i^{[m]}}{\lambda_n-\lambda_{k_i}}\,,
\end{equation}
the left side being independent of $m$. Hence, as $m\to \infty$ we have by Lemma~\ref{lem:trunc} that the prefactor satisfies $\lambda_n-\nu_{j_0}^{[m]}\to\lambda_n-\mu_n=\gamma_n$. For the remaining factors, write for all $0\leq i \leq m$, $i\neq j_0,$
\[
   \frac{\lambda_n-\nu_i^{[m]}}{\lambda_n-\lambda_{k_i}}
   =1+\frac{\lambda_{k_i}-\nu_i^{[m]}}{\lambda_n-\lambda_{k_i}} .
\]
Here $0<\lambda_{k_i}-\nu_i^{[m]}<\gamma_{k_i}$ by \eqref{eq:trunc_interlace}, while $|\lambda_n-\lambda_{k_i}|\ge|n-k_i|\ge1$ since the $\lambda_p$ are strictly increasing with $\lambda_p-\lambda_q\ge p-q$ by \eqref{eq:lambda_lower}. Therefore each factor is positive and
\begin{equation}\label{eq:factor_bound}
    \frac{\lambda_{k_i}-\nu_i^{[m]}}{|\lambda_n-\lambda_{k_i}|}<\gamma_{k_i},
   \qquad 0\le i\le m,\ i\neq j_0 .
\end{equation}
As $\sum_i\gamma_{k_i}\le\sum_p\gamma_p=\|u\|_{L^2}^2<\infty$ by Proposition~\ref{prop:isometry}, we have $\gamma_{k_i}\to0$. Fix $i_0$ with $\gamma_{k_i}\le\tfrac12$ for all $i\ge i_0$. Using $|\log(1+t)|\le2|t|$ for $|t|\le\tfrac12$, \eqref{eq:factor_bound} gives
\begin{equation}\label{eq:log_bound}
   \Bigl|\log\frac{\lambda_n-\nu_i^{[m]}}{\lambda_n-\lambda_{k_i}}\Bigr|\le 2\gamma_{k_i},
   \qquad i\ge i_0,\ i\neq j_0,\ m\ge i .
\end{equation}
The right-hand side of \eqref{eq:log_bound} is independent of $m$ and summable in $i$. Moreover, by Lemma~\ref{lem:trunc}, for each fixed $i\neq j_0$,
\[
   \frac{\lambda_n-\nu_i^{[m]}}{\lambda_n-\lambda_{k_i}}
   \xrightarrow[m\to\infty]{}\frac{\lambda_n-\mu_{k_i}}{\lambda_n-\lambda_{k_i}}
   =\frac{\mu_{k_i}-\lambda_n}{\lambda_{k_i}-\lambda_n}.
\]
Applying the dominated convergence theorem for the series $\sum_{i\neq j_0}\log\frac{\lambda_n-\nu_i^{[m]}}{\lambda_n-\lambda_{k_i}}$ with, by \eqref{eq:log_bound} the summands are dominated for $i\ge i_0$ by $2\gamma_{k_i}$ summable in $i$ and the finitely many terms $i<i_0$ converge directly,
\[
   \sum_{\substack{i=0\\ i\neq j_0}}^{m}\log\frac{\lambda_n-\nu_i^{[m]}}{\lambda_n-\lambda_{k_i}}
   \xrightarrow[m\to\infty]{}
   \sum_{\substack{p\in\K\\ p\neq n}}\log\frac{\mu_p-\lambda_n}{\lambda_p-\lambda_n}.
\]
Exponentiating,
\[
   \prod_{\substack{i=0\\ i\neq j_0}}^{m}\frac{\lambda_n-\nu_i^{[m]}}{\lambda_n-\lambda_{k_i}}
   \xrightarrow[m\to\infty]{}
   \prod_{\substack{p\in\K\\ p\neq n}}\frac{\mu_p-\lambda_n}{\lambda_p-\lambda_n}\in(0,\infty).
\]
Letting $m\to\infty$ in \eqref{eq:kappa_L2_fixedm} yields $|X_n|^2=\gamma_n\prod_{p\in\K,\,p\neq n}\frac{\mu_p-\lambda_n}{\lambda_p-\lambda_n}$. Inactive $p\notin \K$ contribute a factor $\frac{\mu_p-\lambda_n}{\lambda_p-\lambda_n}=1$ (recall $\gamma_p=\lambda_p-\mu_p$, so $\gamma_p=0$ gives $\mu_p=\lambda_p$ ), thus
\[
   \prod_{\substack{p\in\K\\ p\neq n}}\frac{\mu_p-\lambda_n}{\lambda_p-\lambda_n}
   =\prod_{p\neq n}\frac{\mu_p-\lambda_n}{\lambda_p-\lambda_n}=:\kappa_n\,,
   \qtq{and} |X_n|^2=\gamma_n\kappa_n\,.
\]
Each factor of $\kappa_n$ is positive by the interlacing $\mu_p<\lambda_p$ as in Corollary~\ref{cor:kappa}, whence $\kappa_n\in(0,\infty)$.

Now, let us turn to continuity. Let $u_j\to u_*$ in $\Lp$. The set $Q:=\{u_*\}\cup\{u_j:j\ge1\}$ is compact. Each factor $\frac{\mu_p(u)-\lambda_n(u)}{\lambda_p(u)-\lambda_n(u)}$ is continuous on $\Lp$. Indeed, the maps $u\mapsto\lambda_p(u)$ and $\mu_p(u)=\lambda_{p-1}(u)+1$ are continuous by Theorem~\ref{thm:Lax}, and $|\lambda_p(u)-\lambda_n(u)|\ge|p-n|\ge1$. The tails 
$$
    \sum_{p>M}\gamma_p(u)=\|u\|_{L^2}^2-\lambda_0(u)-\sum_{1\le p\le M}\gamma_p(u)
$$
are continuous, nonincreasing in $M$, and $\to0$ pointwise by Proposition~\ref{prop:isometry} as $M\to\infty$. Therefore, applying  Dini's theorem,  $\sum_{p>M}\gamma_p(u)\to0$ uniformly on $Q$. In particular, $\sup_{u\in Q}\gamma_p(u)\to0$, so there is an index $i_0$, uniform on $Q$, beyond which $\gamma_p(u)\le\tfrac12$ and hence, as in \eqref{eq:log_bound}, $\bigl|\log\frac{\mu_p(u)-\lambda_n(u)}{\lambda_p(u)-\lambda_n(u)}\bigr|\le 2\gamma_p(u)$ uniformly on $Q$. With the uniform tail bound, the series $\log\kappa_n(u)=\sum_{p\neq n}\log\frac{\mu_p(u)-\lambda_n(u)}{\lambda_p(u)-\lambda_n(u)}$ converges uniformly on $Q$, being a uniform limit of continuous functions, $\kappa_n$ is continuous on $Q$, i.e.\
$\kappa_n(u_j)\to\kappa_n(u_*)$.
\end{proof}

Finally, we record that the Birkhoff scalings $(\kappa_n)$ are uniformly bounded on $L^2$-balls.

\begin{lemma}[Uniform amplitude bound]\label{lem:kappa_bound}
For every $u\in\Lp$ and $n\ge0$, we have $\ 0<\kappa_n(u)\le e^{\|u\|_{L^2}^2}$. Consequently, for every $N_0\ge1$,
\begin{equation}\label{eq:X_tail}
   \sum_{n>N_0}|X_n|^2\le e^{\|u\|_{L^2}^2}\sum_{n>N_0}\gamma_n .
\end{equation}
\end{lemma}

\begin{proof}
By Corollary~\ref{cor:kappa_L2}, $\kappa_n=\prod_{p\neq n}\frac{\mu_p-\lambda_n}{\lambda_p-\lambda_n}$ with positive factors. For $p>n$, the factor $\frac{\mu_p-\lambda_n}{\lambda_p-\lambda_n}$ lies in $(0,1]$. For $p<n$,
\[
   \frac{\lambda_n-\mu_p}{\lambda_n-\lambda_p}
   =1+\frac{\gamma_p}{\lambda_n-\lambda_p}\le1+\gamma_p .
\]
Hence,  
$$
    \kappa_n\le\prod_{p<n}(1+\gamma_p)\le\prod_{p\ge0}(1+\gamma_p)\le\exp\bigl(\sum_p\gamma_p\bigr) =e^{\|u\|_{L^2}^2}.
$$ 
where the last identity is obtained by Proposition~\ref{prop:isometry}. Then \eqref{eq:X_tail} follows from Corollary~\ref{cor:kappa_L2}: $|X_n|^2=\gamma_n\kappa_n$.
\end{proof}

\subsection{Continuity of the extended Birkhoff map}\label{ss:cont_Phi}

In the transport gauge we define
\begin{equation}\label{eq:zeta_L2}
   \Phi(u):=(\zeta_n(u))_{n\ge0},\qquad \zeta_n(u):=\frac{\langle u\mid f_n\rangle}{\sqrt{\kappa_n}} .
\end{equation}
By Corollary~\ref{cor:kappa_L2}, $\kappa_n\in(0,\infty)$ and $|X_n|^2=\gamma_n\kappa_n$, so
\begin{equation}\label{eq:|zeta[n]|2}
    |\zeta_n(u)|^2=\frac{|X_n|^2}{\kappa_n}=\gamma_n\qquad(n\ge0),
\end{equation}
with the convention $\gamma_0=\lambda_0$. In particular $|\zeta_0|^2=\lambda_0$ and $\zeta_n=0$ whenever
$\gamma_n=0$.

\begin{corollary}\label{cor:ext}
The function $\Phi$ maps $\Lp(\mathbb T)$ into $\ell^2(\N_0)$ and is norm-preserving,
\[
   \|\Phi(u)\|_{\ell^2}^2=\lambda_0+\sum_{n\ge1}\gamma_n=\|u\|_{L^2}^2 .
\]
It extends the finite-gap map of Theorem~\ref{thm:main}: $\Phi(\U_N)=\Omega_N$ for every $N\ge0$ and $\Phi(0)=0$, so $\Phi\bigl(\{0\}\cup\bigcup_{N\ge0}\U_N\bigr)$ is the set of finitely supported sequences, dense in $\ell^2(\N_0)$.
\end{corollary}

\begin{proof}
By \eqref{eq:|zeta[n]|2} and Proposition~\ref{prop:isometry},
\[
   \|\Phi(u)\|_{\ell^2}^2=\sum_{n\ge0}|\zeta_n(u)|^2=\lambda_0+\sum_{n\ge1}\gamma_n=\|u\|_{L^2}^2<\infty,
\]
so $\Phi(u)\in\ell^2(\N_0)$ and $\Phi$ is norm-preserving.

If $u\in\U_N$ then $\gamma_n=0$ for $n>N$ and $\gamma_N>0$, so $\Phi(u)\in\Omega_N$ and the coordinate \eqref{eq:zeta_L2} coincides with \eqref{eq:Phi}.

Finally, let $c_{00}$ denote the finitely supported sequences in $\ell^2(\N_0)$. First, observe $\Phi\bigl(\{0\}\cup\bigcup_{N\ge0}\U_N\bigr)=c_{00}$: the inclusion $\subseteq$ holds since $\Phi(0)=0\in c_{00}$ and $\Phi(\U_N)=\Omega_N\subseteq c_{00}$, conversely, given $\zeta\in c_{00}$ with $\zeta\neq0$, let $N:=\max\{n:\zeta_n\neq0\}\ge0$, so $\zeta\in\Omega_N$, identifying $\Omega_N\subseteq\C^{N+1}$ with its zero-extension, and $\zeta=\Phi(u)$ for some $u\in\U_N$ by Theorem~\ref{thm:main}, while $\zeta=0=\Phi(0)$. Second, $c_{00}$ is dense in $\ell^2(\N_0)$: for $\zeta\in\ell^2$ the truncations $\zeta^{(M)}:=(\zeta_0,\dots,\zeta_M,0,0,\dots)\in c_{00}$ satisfy $\|\zeta-\zeta^{(M)}\|_{\ell^2}^2=\sum_{n>M}|\zeta_n|^2\to0$ as $M\to\infty$. Hence, $\Phi\bigl(\{0\}\cup\bigcup_{N\ge0}\U_N\bigr)$ is dense in $\ell^2(\N_0)$.
\end{proof}

\begin{remark}
Since $\Phi$ is a homeomorphism and $\Omega_N$ is an open subset of $\C^{N+1}$, then the union $\{0\}\cup\bigcup_{N\ge0}\U_N$ i.e., the finite-gap potentials  are dense, since its image under $\Phi$ is the space $c_{00}$ of finitely supported sequences (Corollary~\ref{cor:ext}).
\end{remark}

\begin{lemma}[Continuity of the eigenbasis]\label{lem:eigen_cont}
Under Convention~\textup{\ref{conv:transport}}, for every $n\ge0$ the map $u\mapsto f_n(\cdot,u)$ is continuous from $\Lp$ to $\Lp$. Consequently
$$
    u\mapsto\langle u\mid f_n(\cdot,u)\rangle
$$ 
is continuous on $\Lp(\T)$.
\end{lemma}

\begin{proof} 
Let $u_j\to u$ in $\Lp$. Being convergent, the sequence is bounded, so $u_j,u\in B_r$ for some $r>0$, this is needed to apply Lemma~\ref{lem:nrc}, whose constant depends on such a bound.

Fix $n\ge0$. By \eqref{eq:simplicity_eigenv} the spectrum of $L_u$ is simple. Let 
\[
   \mathscr C:=\bigl\{z\in\C:\ |z-\lambda_n(u)|=\tfrac14\bigr\},\qquad
   P_n(w):=\frac{1}{2\pi i}\oint_{\mathscr C}(z-L_w)^{-1}\,dz .
\] 
 Let us prove that $\mathscr C$ lies in all the resolvent sets of $L_{u}$ and $L_{u_j}$ for $j$ large enough. First, recall from \eqref{eq:gamma[n]} that for any $m\geq 0$, $\la_m-\la_{m-1}-1\geq 0$, so that $\lambda_n(u)$ is the only point of $\sigma(L_u)$ inside $\mathscr C.$ Second, by the continuity of $w\mapsto \la_n(w)$ in Thorem~\ref{thm:Lax}, there exists $j_0\in\N$ such that $|\lambda_n(u_j)-\lambda_n(u)|<\tfrac18$. Therefore
\begin{equation}\label{eq:unif_sep}
   \operatorname{dist}\bigl(\mathscr C,\sigma(L_{u_j})\bigr)\ge\tfrac18\quad(j\ge j_0),
   \qquad
   \operatorname{dist}\bigl(\mathscr C,\sigma(L_{u})\bigr)\ge\tfrac14 ,
\end{equation}
so $\mathscr C$ lies in all the resolvent sets and $\|(L_{u_j}-z)^{-1}\|\le8$, $\|(L_u-z)^{-1}\|\le4$ on $\mathscr C$. Employing dentity \eqref{eq:res_two_sided} with $\kappa=1$, together with the bound \eqref{eq:nrc} and \eqref{eq:unif_sep} in place of \eqref{eq:res_bound_L2}, one obtains
\[
   \sup_{z\in\mathscr C}\bigl\|(L_{u_j}-z)^{-1}-(L_u-z)^{-1}\bigr\|
   \le C\,\|u_j-u\|_{L^2}\xrightarrow[j\to\infty]{}0 ,
\]
$C$ depending only on $\sup_j\|u_j\|_{L^2}$. Integrating over $\mathscr C$ 
gives $\|P_n(u_j)-P_n(u)\|\to0$. Finally $P_n(u)$ is the rank-one orthogonal projection onto $\C f_n(u)$ by simplicity of $\sigma(L_u)$, and two orthogonal projections at distance $<1$ have the same rank. Hence, $\operatorname{rank}P_n(u_j)=1$ for $j$ large, i.e.\ $\lambda_n(u_j)$ is the only eigenvalue of $L_{u_j}$ inside $\mathscr C$ and $\Range P_n(u_j)=\C f_n(u_j)$.
Since $P_n(u)f_n(u)=f_n(u)$, the vectors $P_n(u_j)f_n(u)$ satisfy 
\[
   \|P_n(u_j)f_n(u)-f_n(u)\|\le\|P_n(u_j)-P_n(u)\|\longrightarrow0,
   \quad\text{so}\quad \|P_n(u_j)f_n(u)\|\longrightarrow1 .
\]

Hence, for $j$ large $P_n(u_j)f_n(u)\neq0$ and $P_n(u_j)f_n(u)\in\operatorname{Ran}P_n(u_j)=\C f_n(u_j)$. Thus, $\hat f_{j,n}:=P_n(u_j)f_n(u)/\|P_n(u_j)f_n(u)\|$ is a unit $\lambda_n(u_j)$-eigenvector with
\begin{equation}\label{eq:fj[hat]}
    \hat f_{j,n}\to f_n(u) \text{ in }\Lp \qtq{and} f_n(u_j)=w_{j,n}\hat f_{j,n} \text{ for some }|w_{j,n}|=1.
\end{equation}   

The goal is to prove that for each fixed $n$, $w_{j,n}\to 1$ as $j\to\infty,$ so that $ f_n(u_j)\to f_n(u)$ as $j\to\infty.$ We argue by induction on $n$. For $n=0$, Convention~\ref{conv:transport} requires $\langle1\mid f_0(u_j)\rangle>0$ so $\overline{w_{j,0}}\langle1\mid\hat f_{j,0}\rangle>0$, which means:
$$
    w_{j,0}=\frac{{\langle1\mid\hat f_{j,0}\rangle}}{|\langle1\mid\hat f_{j,0}\rangle|}=\mathrm e^{i \arg \langle 1 \mid \hat{f}_{j,0}\rangle}.
$$
Note that $\langle 1 \mid \hat{f}_{j,0}\rangle$ is nonzero by Lemma~\ref{lem:geometry}(ii) applied to $u_j$. Applying \eqref{eq:fj[hat]} leads to  $\langle1\mid\hat f_{j,0}\rangle\to\langle1\mid f_0(u)\rangle$, which is real and strictly positive. Hence $w_{j,0}\to 1$, so $f_0(u_j)\to f_0(u)$.

Now, let $n\ge1$ and assume $f_{n-1}(u_j)\to f_{n-1}(u)$ in $\Lp$, so that $Sf_{n-1}(u_j)\to Sf_{n-1}(u)$. By Convention~\ref{conv:transport}, $\langle Sf_{n-1}(u_j)\mid f_n(u_j)\rangle=\overline{w_{j,n}}\, \langle Sf_{n-1}(u_j)\mid\hat f_{j,n}\rangle>0$, so, exactly as for $n=0$,
\[
   w_{j,n}=\frac{\langle Sf_{n-1}(u_j)\mid\hat f_j\rangle}
                 {|\langle Sf_{n-1}(u_j)\mid\hat f_j\rangle|} .
\]
By the induction hypothesis and \eqref{eq:fj[hat]}, $\langle Sf_{n-1}(u_j)\mid\hat f_{j,n}\rangle\to\langle Sf_{n-1}(u)\mid f_n(u)\rangle$, which is real and strictly positive by Convention~\ref{conv:transport}. Hence $w_{j,n}\to1$ and $f_n(u_j)=w_{j,n}\hat f_{j,n}\to f_n(u)$, completing the induction.

We conclude that if  $u_j\to u$  then  $f_n(u_j)\to f_n(u)$ in $\Lp$ for any $n\in\N_0$ and so  $\langle u_j\mid f_n(u_j)\rangle\to\langle u\mid f_n(u)\rangle$.
\end{proof}

\begin{remark}[Why Convention~\ref{conv:transport} is used]\label{rem:why_transport}
One might instead normalise the eigenbasis by requiring $\langle u\mid f_n\rangle>0$. On the finite-gap manifolds  $u\in\U_N$ with $\gamma_n>0$ for all $0\le n\le N$, one has $X_n\neq0$ by \eqref{eq:Xn=0}, the phase of $f_n$ is fixed, and the resulting coordinates $\zeta_n=\sqrt{\gamma_n}\,e^{i\arg\langle Sf_{n-1}\mid f_n\rangle}$ ($n\ge1$), $\zeta_0=\sqrt{\gamma_0}\,e^{i\arg\langle1\mid u\rangle}$, again parametrise $\U_N$ bijectively.

This gauge breaks down on all of $\Lp(\T)$, for a structural reason: it is \emph{undefined} precisely at closed gaps. If $\gamma_n=0$ then $X_n=0$ by \eqref{eq:Xn=0}, the condition $\langle u\mid f_n\rangle>0$ is vacuous, and $f_n$ remains undetermined up to a phase. Even when $\gamma_n>0$, letting $\gamma_n\to0$ makes $X_n\to0$ and $\arg X_n$ arbitrarily ill-conditioned, so $f_n$ cannot depend continuously on $u$ near such a configuration.

Convention~\ref{conv:transport} avoids this because the quantities it normalises never vanish: $\langle1\mid f_0\rangle\neq0$ for every $u$ by Lemma~\ref{lem:geometry}(ii), and $\langle Sf_{n-1}\mid f_n\rangle>0$ holds also when $\gamma_n=0$, where $f_n=Sf_{n-1}$ and $\langle Sf_{n-1}\mid f_n\rangle=\|f_{n-1}\|^2=1$. The observable phase is thereby carried by $\langle u\mid f_n\rangle$, which \emph{is} continuous in $u$ (Lemma~\ref{lem:eigen_cont}), and this is what makes Theorem~\ref{thm:continuity} possible.
\end{remark}

\begin{theorem}[Continuity]\label{thm:continuity}
The extended Birkhoff map $\Phi:\Lp(\mathbb T)\to\ell^2(\N_0)$ of \eqref{eq:zeta_L2} is continuous.
\end{theorem}

\begin{proof}
Let $u_j\to u$ in $\Lp$. For each fixed $n$, Lemma~\ref{lem:eigen_cont} gives $\langle u_j\mid f_n(u_j)\rangle\to\langle u\mid f_n(u)\rangle$, and Corollary~\ref{cor:kappa_L2} gives $\kappa_n(u_j)\to\kappa_n(u)\in(0,\infty)$. Hence, for all $n\geq 0,$
\[
   \zeta_n(u_j)=\frac{\langle u_j\mid f_n(u_j)\rangle}{\sqrt{\kappa_n(u_j)}}
   \longrightarrow
   \frac{\langle u\mid f_n(u)\rangle}{\sqrt{\kappa_n(u)}}=\zeta_n(u).
\]
Thus $\Phi(u_j)\to\Phi(u)$ coordinatewise, and by
Corollary~\ref{cor:ext} the sequence $\|\Phi(u_j)\|_{\ell^2}=\|u_j\|_{L^2}$ is bounded, so $\Phi(u_j)\rightharpoonup\Phi(u)$ in $\ell^2$. Finally, 
$$
    \|\Phi(u_j)\|_{\ell^2}=\|u_j\|_{L^2}\to\|u\|_{L^2}=\|\Phi(u)\|_{\ell^2}.
$$
Thus, Radon--Riesz theorem allows us to conclude that  $\Phi(u_j)\to\Phi(u)$ in $\ell^2$, and so $\Phi$ is continuous.
\end{proof}

\subsection{Properness and the homeomorphism}\label{ss:proper}

Continuity and injectivity do not yet give a homeomorphism: $\Phi$ must be shown surjective, and $\Phi^{-1}$ continuous. Both follow from a single compactness property: properness, which we derive from the equicontinuity criterion of \cite{KMV}.

\begin{theorem}[Equicontinuity criterion {\cite[Thm.~3.8]{KMV}}]\label{thm:KMV}
For $u\in\Lp(\T)$ and $\sigma\ge1$ set
\begin{equation}\label{eq:W}
   W(\sigma,u):=\Bigl\langle u\ \Big|\ \frac{(L_u+1)^2}{(L_u+1)^2+\sigma^2}\,u\Bigr\rangle
   =\sum_{n\ge0}\frac{(\lambda_n+1)^2}{(\lambda_n+1)^2+\sigma^2}\,|X_n|^2 .
\end{equation}
For every $r>0$, a subset $Q\subset B_r:=\{u\in\Lp:\|u\|_{L^2}\le r\}$ is $L^2$-equicontinuous, i.e.
$$\lim_{y\to0}\sup_{u\in Q}\|u(\cdot+y)-u\|_{L^2}=0, \qtq{if and only if } \lim_{\sigma\to\infty}\sup_{u\in Q}W(\sigma,u)=0.
$$
In particular, a bounded $L^2$-equicontinuous subset of $\Lp(\T)$ is precompact in $\Lp(\T)$.
\end{theorem}

The proof of this theorem can be found in \cite[Thm.~3.8]{KMV}. The last sentence of the previous theorem is the Kolmogorov--Riesz--Fr\'echet theorem \cite{brezis2011functional} on the compact group $\mathbb T$: a bounded subset of $L^2(\mathbb T)$ is precompact iff it is $L^2$-equicontinuous, and $\Lp$ is closed in $L^2$.

\begin{proposition}[Properness]\label{prop:proper}
The extended Birkhoff map $\Phi:\Lp\to\ell^2(\N_0)$ is proper: If $\Phi(u_k)\to\zeta$ in $\ell^2$, then $(u_k)$ has a subsequence converging in $\Lp$.
\end{proposition}

\begin{proof}
By Corollary~\ref{cor:ext}, $\|u_k\|_{L^2}=\|\Phi(u_k)\|_{\ell^2}\to\|\zeta\|_{\ell^2}$, so $r:=\sup_k\|u_k\|_{L^2}^2<\infty$ and $\{u_k\}\subset B_{\sqrt r}$.  By Theorem~\ref{thm:KMV} it suffices to show $\lim_{\sigma\to\infty}\sup_k W(\sigma,u_k)=0$.

Fix $N_0\ge1$. Since $\lambda\mapsto(\lambda+1)^2/((\lambda+1)^2+\sigma^2)$ is nondecreasing, and by \eqref{eq:lambda_lower} and Proposition~\ref{prop:isometry}, $\lambda_{N_0}(u_k)=\lambda_0(u_k)+N_0+\sum_{n=1}^{N_0}\gamma_n(u_k)\le r+N_0$ , then splitting \eqref{eq:W} at $N_0$ and using $\sum_{n\le N_0}|X_n|^2\le r$, together with  \eqref{eq:X_tail}, 
\[
   W(\sigma,u_k)\le\frac{(\lambda_{N_0}(u_k)+1)^2 \, r}{(\lambda_{N_0}(u_k)+1)^2+\sigma^2}+\sum_{n>N_0}|X_n(u_k)|^2
   \le\frac{(r+N_0+1)^2 \, r}{(r+N_0+1)^2+\sigma^2}+e^{r}\tau_{N_0} ,
\]
where the tails 
\[
   \tau_{N_0}:=\sup_k\sum_{n>N_0}\gamma_n(u_k)=\sup_k\sum_{n>N_0}|\zeta_n(u_k)|^2\xrightarrow[N_0\to\infty]{}0 
\]
uniformly, as $\Phi(u_k)\to\zeta$ in $\ell^2$, and $|\zeta_n(u_k)|^2=\gamma_n(u_k)$ for all $n\ge0$. Therefore, given $\varepsilon>0$, choose $N_0$ with $e^{r}\tau_{N_0}<\varepsilon/2$, then choose $\sigma_*$ such that $\frac{(r+N_0+1)^2\, r}{(r+N_0+1)^2+\sigma_*^2}<\varepsilon/2$ for $\sigma>\sigma_*$ so that $\lim_{\kappa\to\infty}\sup_k W(\sigma,u_k)=0$. By Theorem~\ref{thm:KMV}, $\{u_k\}$ is bounded and $L^2$-equicontinuous, therefore precompact in $\Lp$. Thus, a convergent subsequence exists.
\end{proof}

Although injectivity on the finite-gap manifolds was established in Theorem~\ref{thm:inj}, but that argument used the finiteness of  $\K$ twice: to know that the scalings $\kappa_n$ are finite products determined by the spectrum (Corollary~\ref{cor:kappa}), and to express the $L^2$-norm of $u$ in terms of the $\gamma_n$ (Proposition~\ref{prop:trace}). For general $u\in\Lp$ these are supplied by Corollary~\ref{cor:kappa_L2} and Proposition~\ref{prop:isometry} respectively, and the proof then proceeds as before.

\begin{theorem}[Injectivity on $\Lp$]\label{thm:inj_L2}
$\Phi:u\in \Lp\to (\zeta_n)\in \ell^2(\N_0)$ is injective.
\end{theorem}

\begin{proof}
Let $u,v\in\Lp$ with $\Phi(u)=\Phi(v)=\zeta$.  By \eqref{eq:zeta_L2}, the spectrum $\lambda_n=|\zeta_0|^2+n+\sum_{k=1}^n|\zeta_k|^2$, the active set $\K$, and the weights $\kappa_n$ obtained in Corollary~\ref{cor:kappa_L2} depend only on $\zeta$. Hence, $X_n=\zeta_n\sqrt{\kappa_n}$ and $X(u)=X(v)$, with $\sum_p|X_p|^2=\|u\|_{L^2}^2<\infty$ by Proposition~\ref{prop:isometry}. 

Employing Lemma~\ref{lem:reconstruction} therefore yields $Y(u)=Y(v)$ and $M(u)=M(v)$. As  $\|M\|\le1$, \eqref{eq:hardy_rep} gives $u(z)=v(z)$ on $\D$, i.e.\ $u=v$.
\end{proof}

\begin{theorem}[The Birkhoff map is a homeomorphism]\label{thm:homeo}
The extended Birkhoff map $\Phi:\Lp(\mathbb T)\to\ell^2(\N_0)$ of \eqref{eq:zeta_L2} is a homeomorphism satisfying 
$$
    \|\Phi(u)\|_{\ell^2}=\|u\|_{L^2}.
$$
It extends the finite-gap maps of Theorem~\ref{thm:main}: for each $N\ge0$ it restricts to a bijection $\U_N\to\Omega_N$, and $\Phi(0)=0$.
\end{theorem}

\begin{proof}
Theorem~\ref{thm:continuity} implies the continuity of $\Phi$ and  Theorem~\ref{thm:inj_L2} implies the injectivity.

For the surjectivity, let $\zeta\in\ell^2(\N_0)$. The finitely supported sequences are dense, so choose finitely supported $\zeta^{(k)}\to\zeta$. Each $\zeta^{(k)}$ lies in some $\Omega_{N_k}$ (or is $0$), hence by Theorem~\ref{thm:main} equals $\Phi(u_k)$ with $u_k\in\U_{N_k}$  or $u_k=0$. Then, $\Phi(u_k)\to\zeta$. Employing now Proposition~\ref{prop:proper} gives a subsequence $u_{k_j}\to u$ in $\Lp$ and continuity gives $\Phi(u)=\zeta$. Thus $\Phi$ is a continuous bijection.

Finally, let us prove the continuity of $\Phi^{-1}$. Let $\zeta_k\to\zeta$ in $\ell^2$ and $u_k:=\Phi^{-1}(\zeta_k)$, $u:=\Phi^{-1}(\zeta)$. Every subsequence of $(u_k)$ satisfies $\Phi(u_{k})\to\zeta$, so by Proposition~\ref{prop:proper} it has a further subsequence converging in $\Lp$ to some $u'$ with $\Phi(u')=\zeta$, by injectivity $u'=u$. Hence every subsequence of $(u_k)$ has a further subsequence converging to $u$, so $u_k\to u$ in $\Lp$. Therefore $\Phi^{-1}$ is continuous.

The norm identity is Corollary~\ref{cor:ext}, as is the restriction to $\U_N$.
\end{proof}

\section{Linearization of the flow and quasi-periodicity}\label{S:6}

Having identified $\Phi$ as a norm-preserving homeomorphism $\Lp(\T)\to\ell^2(\N_0)$ (Theorem~\ref{thm:homeo}), we now use it to describe the dynamics of \eqref{CS}. The guiding principle is that $\Phi$ is a system of \emph{Birkhoff coordinates}: in these coordinates the flow becomes an explicit rotation, each modulus $|\zeta_n|$ being a constant of motion and each phase $\arg\zeta_n$ evolving linearly in time at a frequency $\omega_n$ determined by the spectrum alone (Action-Angle variables). We first prove this linearization (Theorem~\ref{thm:linearize}), and then read off its dynamical consequences: the flow map is a norm-preserving homeomorphism of $\Lp$ and every action is conserved, every $L^2$-orbit is precompact and the trajectory $t\mapsto u(t)$ is almost periodic (Corollary~\ref{cor:dynamics}), and finite-gap solutions are quasi-periodic in time with an explicit frequency vector, together with a sharp criterion for  time-periodicity (Corollary~\ref{cor:periodic}).

Let $u\in C(\R;\Lp(\mathbb T))$ be the global solution of \eqref{CS} with datum $u_0$ (see
\cite{badreddine2024global}). Since $t\mapsto L_{u(t)}$ obeys the Lax equation
$\partial_tL_{u(t)}=[P_{u(t)},L_{u(t)}]$, the spectrum is conserved \cite[Corollary 3.2]{badreddine2024global}:
\begin{equation}\label{eq:isospectral}
   \lambda_n(u(t))=\lambda_n(u_0),\qquad \gamma_n(u(t))=\gamma_n(u_0)\qquad t\in\R,\ n\ge0 .
\end{equation}
In particular the Birkhoff scalings $\kappa_n$ of \eqref{eq:kappa_L2} are conserved. Define the
\emph{Birkhoff frequencies}
\begin{equation}\label{eq:frequencies}
   \omega_0:=0 \qtq{and} \quad \omega_n:=\sum_{j=0}^{n-1}(2\lambda_j+1)=n+2\sum_{j=0}^{n-1}\lambda_j,\quad\text{for } n\ge1.
\end{equation}
 By \eqref{eq:isospectral} these are constants of motion.

\begin{theorem}[The Birkhoff map linearizes the flow]\label{thm:linearize}
Let $u\in C(\R;\Lp)$ solve \eqref{CS}. Then, under Convention \eqref{conv:transport}, the spectral coordinates \eqref{eq:zeta_L2} evolve by rotation
\begin{equation}\label{eq:linear_flow}
   \zeta_n(u(t))=e^{-i\omega_n t}\,\zeta_n(u_0),\qquad t\in\R,\ n\ge0 .
\end{equation}
Equivalently, $\Phi$ conjugates \eqref{CS} to the linear flow $\Theta_t:(\zeta_n)\mapsto(e^{-i\omega_n t}\zeta_n)$ on $\ell^2(\N_0)$:
$$\Phi (u(t))=\Theta_t\circ\Phi(u_0).$$
\end{theorem}

\begin{proof}
By \eqref{eq:isospectral}, $|\zeta_n(u(t))|^2=\gamma_n(u(t))=\gamma_n(u_0)=|\zeta_n(u_0)|^2$ and $\kappa_n$ is constant, so only the phase of $\zeta_n$ evolves. It suffices to track $\arg X_n(t)$, where $X_n(t):=\langle u(t)\mid f_n(t)\rangle$ and $(f_n(t))$ is the transport-gauge eigenbasis of $L_{u(t)}$ given by Convention~\ref{conv:transport}.

let $U(t)$ solve
$$
    \partial_tU=P_{u(t)}U \qtq{with} U(0)=\Id.
$$
$U(t)$ is unitary since $P_{u(t)}$ is skew-adjoint, the Lax equation gives $L_{u(t)}=U(t)L_{u_0}U(t)^*$, hence $g_n(t)=U(t)f_n(u_0)$ is a unit eigenvector of $L_{u(t)}$ for $\lambda_n$. Thus, by \cite[Remark~3.1, and Remark 6.1 for the regularity $u\in L^2_+$]{badreddine2024traveling}, the amplitude $(\langle u(t)\mid g_n(t)\rangle$) evolves as
\begin{equation}\label{eq:amp_transport}
   \langle u(t)\mid g_n(t)\rangle=e^{-i\lambda_n^{2}t}\langle u_0\mid f_n(0)\rangle=e^{-i\lambda_n^{2}t}X_n(0).
\end{equation}
Write $g_n(t)=e^{i\theta_n(t)}f_n(t)$ with $\theta_n(0)=0$. We have,
\[
   \langle u(t)\mid g_n(t)\rangle=e^{-i\theta_n(t)}\langle u(t)\mid f_n(t)\rangle=e^{-i\theta_n(t)}X_n(t).
\]
Combining this with \eqref{eq:amp_transport} yields
\begin{equation}\label{eq:phase_relation}
   X_n(t)=e^{\,i\theta_n(t)}\,e^{-i\lambda_n^{2}t}\,X_n(0),
   \qquad\text{so}\qquad
   \arg X_n(t)=\arg X_n(0)+\theta_n(t)-\lambda_n^{2}t.
\end{equation}

It remains to determine the gauge phase $\theta_n(t)$ using \cite[Lemma~3.2]{badreddine2024traveling}: \begin{align}\label{eq:phases}
   \langle 1\mid g_n(t)\rangle & =e^{-i\lambda_n^2 t}\langle1\mid f_n(0)\rangle,\\
   \langle Sg_{n-1}(t)\mid g_n(t)\rangle & =e^{i[(\lambda_{n-1}+1)^2-\lambda_n^2]t}\langle Sf_{n-1}(0)\mid f_n(0)\rangle .\notag
\end{align}
\textit{The base phase $\theta_0(t)$.} By Lemma~\ref{lem:geometry}(ii) and Convention~\ref{conv:transport}, $\langle1\mid f_0(t)\rangle>0$ for every $t$. Writing $g_0(t)=e^{i\theta_0(t)}f_0(t)$ in the first identity of \eqref{eq:phases} gives $e^{-i\theta_0(t)}\langle1\mid f_0(t)\rangle=e^{-i\lambda_0^2t}\langle1\mid f_0(0)\rangle$ with both $\langle1\mid f_0(\cdot)\rangle>0$, whence $e^{-i\theta_0(t)}=e^{-i\lambda_0^2t}$, i.e.\ $\theta_0(t)=\lambda_0^{2}t$ modulo $2\pi$.

\emph{For the transition phases.} For $n\ge1$, the gauge $\langle Sf_{n-1}(t)\mid f_n(t)\rangle>0$ and the second identity in \eqref{eq:phases} give
\[
   e^{-i(\theta_n(t)-\theta_{n-1}(t))}\langle Sf_{n-1}(t)\mid f_n(t)\rangle
   =e^{i[(\lambda_{n-1}+1)^2-\lambda_n^2]t}\langle Sf_{n-1}(0)\mid f_n(0)\rangle,
\]
with both inner products positive for every $t$ by Convention \ref{conv:transport}. Hence, $\theta_n(t)-\theta_{n-1}(t)=-[(\lambda_{n-1}+1)^2-\lambda_n^2]t$ . Telescoping the transition phases with the base phase $\theta_0=\lambda_0^2 t$:
\begin{align*}
   \theta_n(t)&=\lambda_0^2 t+\sum_{k=1}^{n}\bigl[\lambda_k^2-(\lambda_{k-1}+1)^2\bigr]t \\
   &=t\Bigl(\lambda_n^2-\sum_{j=0}^{n-1}\bigl[(\lambda_j+1)^2-\lambda_j^2\bigr]\Bigr)
   =t\Bigl(\lambda_n^2-\sum_{j=0}^{n-1}(2\lambda_j+1)\Bigr).
\end{align*}
Substituting this into \eqref{eq:phase_relation},
\[
   \arg X_n(t)-\arg X_n(0)=\theta_n(t)-\lambda_n^{2}t=-\,t\sum_{j=0}^{n-1}(2\lambda_j+1)=-\omega_n t,
\]
where all phase identities above are understood modulo $2\pi$. As $\kappa_n$ is constant, 
\[
   \zeta_n(u(t))=\frac{X_n(t)}{\sqrt{\kappa_n}}=e^{-i\omega_n t}\,\frac{X_n(0)}{\sqrt{\kappa_n}}
   =e^{-i\omega_n t}\zeta_n(u_0).
\]
For $\gamma_n=0$ both sides vanish. This proves \eqref{eq:linear_flow}.
\end{proof}

\medskip
Before proving the almost periodicity of the flow,  we record two elementary known facts used in its proof: one about the geometry of $\ell^2(\N_0)$, one about (Bohr) almost periodic functions.

\begin{lemma}[The invariant torus $\operatorname{Tor}\zeta$]\label{lem:torus_compact}
For $\zeta=(\zeta_n)_{n\ge0}\in\ell^2(\N_0)$ define
\begin{equation}\label{eq:torus_def}
   \operatorname{Tor}\zeta:=\bigl\{\eta=(\eta_n)_{n\ge0}\in\C^{\N_0}:\ |\eta_n|=|\zeta_n|\ \ \forall n\ge0\bigr\}.
\end{equation}
Then,
\begin{enumerate}
   \item[\textup{(a)}] $\operatorname{Tor}\zeta\subset\ell^2(\N_0)$, and $\|\eta\|_{\ell^2}=\|\zeta\|_{\ell^2}$ for every $\eta\in\operatorname{Tor}\zeta$.
   \item[\textup{(b)}] $\operatorname{Tor}\zeta$ is compact in the norm topology of $\ell^2(\N_0)$.
   \item[\textup{(c)}] $\Theta_t\zeta\in\operatorname{Tor}\zeta$ for every $t\in\R$, where $\Theta_t$ is as in Theorem~\ref{thm:linearize}.
\end{enumerate}
\end{lemma}

\begin{proof}
(a) and (c) are immediate: every $\eta\in\operatorname{Tor}\zeta$ has $|\eta_n|=|\zeta_n|$, so $\|\eta\|_{\ell^2}^2=\sum_n|\zeta_n|^2=\|\zeta\|_{\ell^2}^2<\infty$ and $|(\Theta_t\zeta)_n|=|e^{-i\omega_nt}\zeta_n|=|\zeta_n|$ for every $n$, so $\Theta_t\zeta\in\operatorname{Tor}\zeta$.

For (b), $\operatorname{Tor}\zeta$ is closed in $\ell^2$: if $\eta^{(k)}\to\eta$ in $\ell^2$ with $\eta^{(k)}\in\operatorname{Tor}\zeta$, then $\eta^{(k)}_n\to\eta_n$ for each fixed $n$, so $|\eta_n|=\lim_k|\eta^{(k)}_n|=|\zeta_n|$, i.e. $\eta\in\operatorname{Tor}\zeta$.

We show $\operatorname{Tor}\zeta$ is \emph{totally} bounded, being closed in the complete space $\ell^2(\N_0)$ this gives compactness \cite{rudin2021principles, brezis2011functional}. Fix $\eps>0$ and choose $N$ with $\sum_{n>N}|\zeta_n|^2<\eps^2/4$ (possible since $\zeta\in\ell^2$). The finite product $\prod_{n=0}^{N}\{\eta_n\in\C:|\eta_n|=|\zeta_n|\}$ is a compact subset of $\C^{N+1}$, hence admits a finite $\tfrac{\eps}{2}$-net $\eta^{(1)},\dots,\eta^{(J)}$ (for the Euclidean norm on $\C^{N+1}$). Extend each $\eta^{(j)}$ by zero to an element of $\ell^2(\N_0)$. Given any $\eta\in\operatorname{Tor}\zeta$, write $\eta=(\eta_{\le N},\eta_{>N})$, choose $j$ with $\|\eta_{\le N}-\eta^{(j)}_{\le N}\|<\eps/2$. Since $\|\eta_{>N}\|_{\ell^2}=(\sum_{n>N}|\zeta_n|^2)^{1/2}<\eps/2$ and $\eta^{(j)}_{>N}=0$,
\[
   \|\eta-\eta^{(j)}\|_{\ell^2}\le\|\eta_{\le N}-\eta^{(j)}_{\le N}\|+\|\eta_{>N}\|<\eps .
\]
Thus $\{\eta^{(1)},\dots,\eta^{(J)}\}$ is an $\eps$-net for $\operatorname{Tor}\zeta$ in $\ell^2(\N_0)$, proving total
boundedness.
\end{proof}

\begin{lemma}[\cite{amerio2013almost} Permanence of almost periodicity]\label{lem:ap_permanence}
Let $X,Y$ be Banach spaces. 
\begin{enumerate}
   \item[\textup{(a)}] If $f_k:\R\to X$ are almost periodic and $f_k\to f$ uniformly on $\R$, then $f$ is almost periodic.
   \item[\textup{(b)}] If $F\subset X$ is compact, $f:\R\to F$ is almost periodic and $g:F\to Y$ is continuous, then $g\circ f:\R\to Y$ is almost periodic.
\end{enumerate}
\end{lemma}

Based on the linearization of the flow of Theorem~\ref{S:6}, we deduce the following.

\begin{corollary}[Global dynamics via the Birkhoff map]\label{cor:dynamics}
Let $u_0\in\Lp(\mathbb T)$ and $\zeta:=\Phi(u_0)$. Then the solution of \eqref{CS} is
\begin{equation}\label{eq:explicit_solution}
   u(t)=\Phi^{-1}\bigl(\Theta_t\zeta\bigr)=\Phi^{-1}\bigl((e^{-i\omega_n t}\zeta_n)_{n\ge0}\bigr),
   \qquad t\in\R .
\end{equation}
where $\Theta_t$ is defined in Theorem~\ref{thm:linearize}. Consequently:
\begin{enumerate}
   \item[\textup{(i)}] \emph{(Solution map and conservation laws.)} For every $t\in\R$, the map
   $S(t):u_0\mapsto u(t)$ is a norm-preserving homeomorphism of $\Lp$ , $\|u(t)\|_{L^2}=\|u_0\|_{L^2}$,
   and $t\mapsto u(t)$ is continuous from $\R$ to $\Lp$. In addition,  all the actions $|\zeta_n|^2=\gamma_n$ are conserved by the flow.
   \item[\textup{(ii)}] \emph{(Precompactness and Almost-periodicity.)} For every $u_0\in\Lp$ the orbit $\{u(t):t\in\R\}$ is precompact in $\Lp(\T)$, and $t\mapsto u(t)$ is (Bohr) almost periodic.
\end{enumerate}
\end{corollary}

\begin{proof}
Formula \eqref{eq:explicit_solution} is Theorem~\ref{thm:linearize} composed with the homeomorphism
$\Phi^{-1}$ of Theorem~\ref{thm:homeo}.

\medskip
First, let us prove (i). The operator $\Theta_t$ is the diagonal operator on $\ell^2(\N_0)$ with entries $e^{-i\omega_nt}$. Hence, $\Theta_t^{*}\Theta_t=\Theta_t\Theta_t^{*}=\mathrm{Id}$, i.e.\ $\Theta_t$ is unitary on $\ell^2(\N_0)$ with continuous inverse $\Theta_t^{-1}=\Theta_{-t}$. Consequently $S(t)=\Phi^{-1}\circ\Theta_t\circ\Phi$ is a composition of three homeomorphisms $\Phi:\Lp\to\ell^2$, $\Theta_t:\ell^2\to\ell^2$, $\Phi^{-1}:\ell^2\to\Lp$, hence is itself a homeomorphism of $\Lp$, with inverse $S(t)^{-1}=\Phi^{-1}\circ\Theta_{-t}\circ\Phi=S(-t)$. Moreover, each of the three maps is norm-preserving by Theorem~\ref{thm:homeo}, so
\[
   \|u(t)\|_{L^2}=\|\Phi^{-1}(\Theta_t\Phi(u_0))\|_{L^2}=\|\Theta_t\Phi(u_0)\|_{\ell^2}
   =\|\Phi(u_0)\|_{\ell^2}=\|u_0\|_{L^2}.
\]
The group property $S(t+s)=S(t)S(s)$ follows from $\Theta_{t+s}=\Theta_t\Theta_s$  and $S(0)=\mathrm{Id}$ from $\Theta_0=\mathrm{Id}$.

For continuity of $t\mapsto u(t)=\Phi^{-1}(\Theta_t\zeta)$, fix $t_0\in\R$ and let $t\to t_0$. For each $n$, $e^{-i\omega_nt}\to e^{-i\omega_nt_0}$, and
\[
   \|\Theta_t\zeta-\Theta_{t_0}\zeta\|_{\ell^2}^2=\sum_{n\ge0}|\zeta_n|^2\,\bigl|e^{-i\omega_nt}-e^{-i\omega_nt_0}\bigr|^2 .
\]
Each summand is bounded by $4|\zeta_n|^2$ (summable) and tends to $0$ as $t\to t_0$. By dominated convergence $\|\Theta_t\zeta-\Theta_{t_0}\zeta\|_{\ell^2}\to0$, i.e.\ $t\mapsto\Theta_t\zeta$ is continuous $\R\to\ell^2$. Composing with the continuous map $\Phi^{-1}$ gives $t\mapsto u(t)$ continuous.

\medskip
Now, let us prove (ii). By Lemma~\ref{lem:torus_compact}, the whole curve $\{\Theta_t\zeta:t\in\R\}$ is contained in the  set $\operatorname{Tor}\zeta\subset\ell^2(\N_0)$ where $\operatorname{Tor}\zeta$ is a compact subset of $\ell^2(\N_0)$ by Lemma~\ref{lem:torus_compact}(b). Since $\Phi^{-1} : \ell^2(\N_0) \to \Lp$ is continuous, its image $\Phi^{-1}(\operatorname{Tor}\zeta)$ is compact in $\Lp$. The orbit $\{u(t) : t \in \R\} = \Phi^{-1}(\{\Theta_t\zeta : t \in \R\})$ is contained in $\Phi^{-1}(\operatorname{Tor}\zeta)$, which directly proves its precompactness in $\Lp$.

For almost periodicity, first note that for each $M\ge0$ the truncated curve
\[
   t\longmapsto\Theta_t^{[M]}\zeta:=\sum_{n=0}^{M}\zeta_n\,e^{-i\omega_nt}\,e_n \in\ell^2(\N_0)
\]
 is an $\ell^2(\N_0)$--valued trigonometric polynomial, hence (Bohr) almost periodic. Moreover,
\[
   \sup_{t\in\R}\bigl\|\Theta_t\zeta-\Theta_t^{[M]}\zeta\bigr\|_{\ell^2}^2
   =\sup_{t\in\R}\sum_{n>M}|\zeta_n|^2=\sum_{n>M}|\zeta_n|^2\xrightarrow[M\to\infty]{}0,
\]
since $\zeta\in\ell^2$. That is, $\Theta_t^{[M]}\zeta\to\Theta_t\zeta$ uniformly in $t$. By Lemma~\ref{lem:ap_permanence}(a), $t\mapsto\Theta_t\zeta$ is almost periodic in $\ell^2(\N_0)$.

Finally, $t\mapsto\Theta_t\zeta$ takes values in the compact set $\operatorname{Tor}\zeta$, and $\Phi^{-1}:\ell^2\to\Lp$ is continuous by Theorem~\ref{thm:homeo}. Thus, by Lemma~\ref{lem:ap_permanence}(b) applied with $F=\operatorname{Tor}\zeta$, $f(t)=\Theta_t\zeta$, $g=\Phi^{-1}|_{\operatorname{Tor}\zeta}$, the composition $u(t)=\Phi^{-1}(\Theta_t\zeta)$ is almost periodic.
\end{proof}

\begin{corollary}[Quasi-periodicity and Periodicity criterion for finite-gap data]\label{cor:periodic} 
Let $u_0\in\U_N$, $\zeta=\Phi(u_0)$ and set $\Ks:=\{1\le n\le N:\zeta_n\neq0\}$ with $d:=|\Ks|$. Then,
\begin{enumerate}
   \item[\textup{(i)}] There is a continuous map $F:\T^{d}\to\U_N\subset L^2_+(\T)$ with
   \[
      u(t)=F\bigl((-\omega_n t\ \mathrm{mod}\ 2\pi)_{n\in\Ks}\bigr),\qquad t\in\R.
   \]
   In particular $u$ is quasi-periodic with frequency vector $(\omega_n)_{n\in\Ks}$ and $u\equiv u_0$ if $d=0$ (i.e.\ $N=0$).
   \item[\textup{(ii)}] The solution $t\mapsto u(t)$ is periodic if and only if the frequencies $(\omega_n)_{n\in\Ks}$ satisfy 
   $$
    \omega_n/\omega_m\in\mathbb Q\qtq{ for all }n,m\in\Ks.
    $$
    In particular, if $\Ks=\{n_0\}$ is a
   singleton, then $u$ is $\tfrac{2\pi}{\omega_{n_0}}$-periodic in time. This is the one-gap
   (traveling-wave) case of \cite{badreddine2024traveling}.
   \end{enumerate}
\end{corollary}

\begin{proof}
(i). As $u_0\in\U_N$, Theorem~\ref{thm:main} gives $\zeta=\Phi(u_0)\in\Omega_N$, i.e. $\zeta$ is supported in $\{0,\dots,N\}$ with $\zeta_N\neq0$. Thus, $\K\subseteq\{0,\dots,N\}$ is finite and nonempty, so $d:=|\Ks|\ge1$. Identify $\T^d=(\R/2\pi\Z)^{\Ks}$ and define the parametrization
\[
    \Gamma:\T^d\to\Omega_N,\qquad
   \Gamma(\varphi)_n:=
   \begin{cases}
      \zeta_n\,e^{i\varphi_n} & n\in\Ks\\
      \zeta_0 & n=0\\
      0 & \text{otherwise}
   \end{cases},
\]
which is well defined and continuous. Set
\[
   F:=\Phi^{-1}\circ\Gamma:\T^d\longrightarrow\U_N ,
\]
continuous by Theorem~\ref{thm:homeo}. By definition of $\Theta_t$ and since $\zeta_n=0$ for $n\notin\K$,
\[
   (\Theta_t\zeta)_n=
   \begin{cases}
      \zeta_n\,e^{-i\omega_nt}, & n\in\K,\\
      0, & n\notin\K,
   \end{cases}
   \qquad\text{i.e.}\qquad
   \Theta_t\zeta=\Gamma\bigl((-\omega_nt\ \mathrm{mod}\ 2\pi)_{n\in\Ks}\bigr).
\]
Hence, by \eqref{eq:explicit_solution},
\[
   u(t)=\Phi^{-1}(\Theta_t\zeta)=F\bigl((-\omega_nt\ \mathrm{mod}\ 2\pi)_{n\in\Ks}\bigr),
\]
which is precisely the statement that $u$ is quasi-periodic with frequencies $(\omega_n)_{n\in\Ks}$.

\medskip
For (ii).
If $d=0$ then $\Ks=\emptyset$, $u\equiv u_0$ is constant, hence periodic. Assume therefore $d\ge1$. Since $\Phi$ is a homeomorphism by Theorem~\ref{thm:homeo}  and, by \eqref{eq:explicit_solution}, $\Phi(u(t))=\Theta_t\zeta$ for any $T>0$, the following  equivalences hold
\begin{align*}
    u(t+T)=u(t)\ \ \forall t\in\R
   &\iff \Theta_{t+T}\zeta=\Theta_t\zeta\ \ \forall t\in\R\\
   &\iff e^{-i(t+T)\cdot\,\omega_n}\zeta_n=e^{-i\omega_n t}\zeta_n\ \ \forall t,\ \forall n .
\end{align*}
For $n\notin\K$ we have $\zeta_n=0$ and the last identity is automatic. For $n=0$, $\omega_0=0$ and it is again automatic. Dividing by $\zeta_n\neq0$ for $n\in\Ks$, the factor $e^{-i\omega_n t}$ cancels, and we obtain
\begin{equation}\label{eq:period_condition}
   u\ \text{is $T$-periodic}\iff e^{-i\omega_n T}=1\ \ \forall n\in\Ks
   \iff \omega_n T\in2\pi\Z\ \ \forall n\in\Ks .
\end{equation}
Thus $u$ is periodic if and only if there exists $T>0$ with $\omega_n T\in2\pi\Z$ for every $n\in\Ks$. We show that such a $T$ exists precisely when the frequencies $(\omega_n)_{n\in\Ks}$ satisfy the condition: $\frac{\omega_n}{\omega_m}\in\mathbb{Q}$ pairwise.

$(\Rightarrow)$ Suppose $T>0$ satisfies \eqref{eq:period_condition}, say $\omega_n T=2\pi k_n$ with $k_n\in\Z$. As $\omega_n>0$ and $T>0$, each $k_n\in\N$. Hence, for all $n,m\in\Ks$,
\[
   \frac{\omega_n}{\omega_m}=\frac{k_n}{k_m}\in\mathbb Q .
\]

$(\Leftarrow)$ Conversely, suppose $\omega_n/\omega_m\in\mathbb Q$ for all $n,m\in\Ks$. Fix $m_0\in\Ks$ and write, for each $n\in\Ks$, $\omega_n/\omega_{m_0}=p_n/q_n$ in lowest terms with $p_n,q_n\in\N$. Since $\Ks$ is finite, $Q:=\operatorname{lcm}\{q_n:n\in\Ks\}$ is a well-defined positive integer. Set
\[
   T:=\frac{2\pi Q}{\omega_{m_0}}>0 .
\]
Then, for every $n\in\Ks$,
\[
   \omega_n T=2\pi Q\,\frac{\omega_n}{\omega_{m_0}}=2\pi Q\,\frac{p_n}{q_n}
   =2\pi\,p_n\,\frac{Q}{q_n}\in2\pi\Z ,
\]
because $q_n\mid Q$. By \eqref{eq:period_condition}, $u$ is $T$-periodic.

Finally, if $\Ks=\{n_0\}$ is a singleton the pairwise condition is vacuous, so $u$ is periodic. taking $Q=1$, $m_0=n_0$ above gives the period $T=2\pi/\omega_{n_0}$. This is moreover the fundamental period: the curve $t\mapsto\Theta_t\zeta$ has only the single nonconstant coordinate $\zeta_{n_0}e^{-i\omega_{n_0}t}$ (with $\zeta_{n_0}\neq0$), whose minimal period is $2\pi/\omega_{n_0}$, and $\Phi^{-1}$ is injective. This is the one-gap (traveling-wave) regime of \cite{badreddine2024traveling}.
\end{proof}
\bibliographystyle{amsplain}
\bibliography{CMref}

\end{document}